\documentclass[11pt, a4paper]{article}
\usepackage[utf8]{inputenc}
\usepackage{mathtools}
\usepackage{amsmath}
\usepackage{amsthm}
\usepackage{algorithm}
\usepackage{algorithmic}
\usepackage{amssymb}
\usepackage{parskip}
\usepackage{stackengine}
\mathtoolsset{showonlyrefs}
\usepackage[dvipsnames]{xcolor}
\usepackage{tikz}
\usepackage{pgfplots}
\pgfplotsset{major grid style={dashed}}
\pgfplotsset{compat=1.18}
\usepackage[a4paper, margin=2.5cm]{geometry}
\usepackage{comment}
\usepackage{mathrsfs} 
\usepackage{dsfont}
\usepackage{esint}
\usepackage[export]{adjustbox}
\usepackage{multirow}
\usepackage{mathtools}
\usepackage{enumitem}
\usepackage{scalerel}
\usepackage{color}
\definecolor{hanblue}{rgb}{0.27, 0.42, 0.81}
\definecolor{mordantred19}{rgb}{0.68, 0.05, 0.0}
\definecolor{red}{rgb}{0.68, 0.05, 0.0}
\definecolor{green}{rgb}{0.0, 0.5, 0.0}
\usepackage{cite}
\usepackage[colorlinks, citecolor=hanblue,linkcolor=green]{hyperref}

\DeclareMathOperator*{\argmin}{arg\,min}

\newcommand{\abs}[1]{|#1|}

\DeclareMathOperator{\supp}{supp}

\newcommand{\R}{\mathbb{R}}
\newcommand{\Z}{\mathbb{Z}}

\newcommand{\N}{\mathbb{N}}
\renewcommand{\S}{\mathbb{S}}

\newcommand{\1}{\mathds 1}
\renewcommand{\L}{\mathcal{L}}
\newcommand{\M}{\mathcal{M}}

\newcommand{\Ha}{\mathcal{H}}

\newcommand{\TV}{\operatorname{TV}}

\newcommand{\BV}{\operatorname{BV}}

\newcommand{\Per}{\operatorname{P}}
\newcommand{\GPer}{\mathbb{P}}
\newcommand{\Ext}{\operatorname{Ext}}

\renewcommand{\div}{\operatorname{div}}

\DeclareMathOperator*{\essinf}{ess\,inf}
\newcommand{\dd}{\, \mathrm{d}}
\DeclareMathOperator{\esssupp}{ess\,supp}

\newcommand{\res}{\mathop{\hbox{\vrule height 7pt width .5pt depth 0pt
\vrule height .5pt width 6pt depth 0pt}}\nolimits}
\newcommand{\mres}{\mathbin{\vrule height 1.4ex depth 0pt width
0.13ex\vrule height 0.13ex depth 0pt width 1.0ex}}

\newcommand\restr[2]{{
  \left.\kern-\nulldelimiterspace
  #1 
  \vphantom{\big|} 
  \right|_{#2}
  }}

\allowdisplaybreaks

\theoremstyle{plain}
\newtheorem{thm}{Theorem}
\numberwithin{thm}{section}

\newtheorem{lemma}[thm]{Lemma}
\newtheorem{prop}[thm]{Proposition}

\newtheorem{cor}[thm]{Corollary}

\theoremstyle{definition}
\newtheorem{defi}[thm]{Definition}

\newtheorem{remarkx}[thm]{Remark}
\newenvironment{rem}
  {\pushQED{\qed}\remarkx}
  {\popQED\endremarkx}
\newtheorem{examplex}[thm]{Example}
\newenvironment{example}
  {\pushQED{\qed}\examplex}
  {\popQED\endexamplex}
\theoremstyle{remark}

\renewcommand{\epsilon}{\varepsilon}
\newcommand{\Qcal}{\mathcal{Q}}
\newcommand{\Pcal}{\mathcal{P}}
\newcommand{\Scal}{\mathcal{S}}
\newcommand{\Acal}{\mathcal{A}}
\newcommand{\Bcal}{\mathcal{B}}
\newcommand{\Hcal}{\mathcal{H}}
\newcommand{\Lcal}{\mathcal{L}}
\newcommand{\Mcal}{\mathcal{M}}
\newcommand{\Fcal}{\mathcal{F}}
\DeclareMathOperator{\dist}{dist}
\DeclareMathOperator{\Div}{div}
\renewcommand{\phi}{\varphi}
\renewcommand{\leq}{\leqslant}
\renewcommand{\geq}{\geqslant}

\title{Nonlocal perimeters and variations: Extremality and decomposability for finite and infinite horizons}
\author{ Marcello Carioni\thanks{Department of Applied Mathematics, University of Twente, 7500AE Enschede, The Netherlands \\
(\texttt{m.c.carioni@utwente.nl}, \texttt{l.delgrande{@}utwente.nl}, \texttt{jose.iglesias@utwente.nl})},\ \ Leonardo Del Grande\footnotemark[1],\ \ José A. Iglesias\footnotemark[1],\ \ Hidde Sch{\"o}nberger\thanks{Research Institute in Mathematics and Physics, Universit\'{e} catholique de Louvain, Chemin du Cyclotron 2, 1348 Louvain-la-Neuve, Belgium (\texttt{hidde.schonberger@uclouvain.be})} }
\date{}

\begin{document}

\maketitle

\begin{abstract}
\noindent
We analyze the extremality and decomposability properties with respect to two types of nonlocal perimeters available in the literature, the Gagliardo perimeter based on the eponymous seminorms and the nonlocal distributional Caccioppoli perimeter, both with finite and infinite interaction ranges. 
A nonlocal notion of indecomposability associated to these perimeters is introduced, and we prove that in both cases it can be characterized solely in terms of the interaction range or horizon $\varepsilon$. Utilizing this, we show that it is possible to uniquely decompose a set into its $\varepsilon$-connected components, establishing a nonlocal analogue of the decomposition theorem of Ambrosio, Caselles, Masnou and Morel. 
Moreover, the extreme points of the balls induced by the Gagliardo and nonlocal total variation seminorm are identified, which naturally correspond to the two nonlocal perimeters. Surprisingly, while the extreme points in the former case are normalized indicator functions of $\varepsilon$-simple sets, akin to the classical TV-ball, in the latter case they are instead obtained from a nonlocal transformation applied to the extreme points of the TV-ball. Finally, we explore the nonlocal-to-local transition via a $\Gamma$-limit as $\varepsilon \rightarrow 0$ for both perimeters, recovering the classical Caccioppoli perimeter.
\end{abstract}


\section{Introduction}

In recent years, nonlocal functionals have been extensively studied as a counterpart to local ones, offering a powerful framework for capturing interactions that extend over long distances without relying on gradient-based quantities. This distinctive feature has made nonlocal functionals particularly appealing in contexts where traditional local models may fall short. As a result, a wide range of nonlocal analogs to classical concepts have been developed and thoroughly analyzed, finding application across various domains such as image processing \cite{antil2017spectral, bartels2020parameter, davoli2023structural, iglesias2022convergence, AntDiaJinSch24, bessas2025non}, where they enhance edge detection and denoising techniques, data science \cite{BunTriMur23,BunKer24, antil2020fractional}, where capturing global relationships between data points is critical, and in mechanical models \cite{mengesha2015variational, bellido2023non, bellido2015hyperelasticity}, which benefit from their ability to describe phenomena involving long-range interactions. Due to the plethora of different nonlocal functionals that are available in the literature, we dedicate the next part of the introduction to describe those that are relevant for the present paper.

\textbf{Nonlocal Sobolev spaces, gradients, perimeters and BV functions.} The classical concept of fractional Sobolev spaces \cite{DinPalVal12} (also known as Gagliardo–Slobodeckij spaces) denoted by $W^{\alpha,p}(\R^d)$ for $\alpha \in (0,1)$ and $1 \leqslant p < \infty$, is defined as the space of $L^p(\R^d)$ functions for which the so-called Gagliardo–Slobodeckij (or fractional) seminorm
\[|u|^p_{W^{\alpha,p}(\R^d)} := \int_{\R^d} \int_{\R^d} \frac{|u(x)-u(y)|^p}{|x-y|^{d+\alpha p}} \dd x \dd y\]
is finite. This notion provides a foundation for the definition of the fractional perimeter, given by
\begin{align}\label{eq:fracperintro}
    \GPer_\alpha(E) := \big|\1_E\big|_{W^{\alpha,1}(\R^d)},
\end{align}
that has been extensively investigated as a nonlocal version of the classical Caccioppoli perimeter. Celebrated results on nonlocal minimal surfaces \cite{CafRoqSav10, CesDipNovVal18} and nonlocal free-discontinuity problems \cite{frank2008non,figalli2015isoperimetry} have extended well-known classical results, shedding light on the peculiar properties of fractional perimeters.
Moreover, in recent years, the concept of fractional perimeter described above, along with that of fractional Sobolev spaces, has been further generalized by exploring kernels with properties that differ from the standard fractional kernel $\rho_\alpha(x-y) = |x-y|^{-d+1-\alpha}$. This has led to the development of a more general notion of the Gagliardo-Slobodeckij space, $W^{\rho,p}(\R^d)$, and nonlocal perimeters, named Gagliardo perimeters, whose properties heavily depend on the chosen kernel $\rho$ \cite{CesNov17, CesNov18, berendsen2019asymptotic}. 

The fractional Sobolev spaces $W^{\alpha,p}(\R^d)$ do not directly arise from the definition of a fractional gradient. Indeed, if one considers the Riesz fractional gradient introduced in its seminal form in \cite{Hor59} and revitalized in \cite{shieh2015new,shieh2018new}, that is,
\[D^\alpha u(x) := c_{d, \alpha} \int_{\R^d} \frac{u(y)-u(x)}{\abs{y-x}^{d+\alpha}}\frac{y-x}{\abs{y-x}}\dd y \quad \text { with } \quad c_{d, \alpha}:=2^\alpha \pi^{-d / 2} \frac{\Gamma((d+\alpha+1) / 2)}{\Gamma((1-\alpha) / 2)},
\]
one can observe that the relation between $\|D^\alpha u\|_{L^p(\R^d)}$ and $|u|_{W^{\alpha,p}(\R^d)}$, and thus the corresponding embeddings, are far from straightforward. In fact, the natural spaces associated to the fractional gradient are the Bessel potential spaces $H^{\alpha,p}(\R^d)$ which are different from $W^{\alpha,p}(\R^d)$ when $p \not =2$, cf.~\cite{shieh2015new}. The fractional gradients have been successfully used to study nonlocal variational models \cite{shieh2015new,shieh2018new}, with applications in hyperelasticity \cite{bellido2023non} and image processing \cite{AntDiaJinSch24}.

Moreover, while the Gagliardo-Slobodeckij spaces are not suitable for defining a nonlocal notion of bounded variation (BV) function, this shortcoming has been addressed in \cite{ComSte19, ComSte23, brue2022distributional}, where a fractional BV space, denoted by $\BV^\alpha(\R^d)$, is defined through the fractional divergence operator 
\[
\div^\alpha p(x):= c_{d,\alpha}\int_{\R^d} \frac{p(y)-p(x)}{\abs{y-x}^{d+\alpha}}\cdot\frac{y-x}{\abs{y-x}} \dd y
\]
as all $L^1(\R^d)$ functions such that 
\begin{align}\label{eq:disBVintro}
\TV_\alpha (u):=\sup\left\{ \int_{\R^d} u \div^\alpha p \dd x \,\middle\vert\, p \in C_c^{\infty}(\R^d;\R^d),\, \|p\|_{L^{\infty}(\R^d;\R^d)}\leqslant 1\right\}<\infty.
\end{align}
Associated to the fractional variation is a notion of fractional Caccioppoli perimeter defined as
\begin{align}\label{eq:distrperintro}
\Per_\alpha(E) = \TV_\alpha(\1_E),  
\end{align}
for all measurable sets $E\subset \R^d$. The properties of $\Per_\alpha(E)$ and the relation to the more classical $\GPer_\alpha(E)$ are nowadays not known in full generality and remain active research directions. We refer to \cite{ComSte19,ComSte23, brue2022distributional} for the known properties and relations between different types of perimeters.

\textbf{Geometry of the total variation ball: extremality and decomposability.} The geometry of the total variation ball for BV functions has been carefully studied in the last years, shedding light on properties of variational problems involving total variation (TV) norms and perimeter penalization. 
Building on the foundational works by Federer \cite{federer2014geometric} and Fleming \cite{fleming60, Fle57}, the paper \cite{AmbCasMasMor01} deeply analyzes geometric properties of the TV-ball, pointing out the importance of indecomposable sets of finite perimeters to derive a natural notion of measure-theoretical connectedness and to characterize the extreme points of the total variation unit ball $\mathcal{B}_{\TV}$ (see also \cite{BonGus22, bredies2020sparsity}) as follows:
\begin{align}\label{eq:extbvintro}
\Ext(\mathcal{B}_{\TV}) = \left\{\pm \frac{\1_E}{\Per(E)} \,\middle|\, \Per(E) < \infty, \ E \ \text{is simple}\right\}.    
\end{align}
Moreover, further advances have been made in the study of the faces of the total variation unit ball in \cite{duval2022faces}. These results have significantly advanced the understanding of the geometry of the total variation ball for BV functions, paving the way for the development of a sparsity theory for TV-regularized optimization problems. Indeed, recent research trends have demonstrated that knowledge of extreme points can be leveraged to establish representer theorems \cite{boyer2019representer, bredies2020sparsity}, analyze the stability of sparsity under perturbations \cite{de2024exact, carioni2023general}, and, ultimately, design sparse optimization algorithms \cite{de2023towards, bredies2024asymptotic, cristinelli2023conditional}.  
Despite extensions to higher-order TV energies \cite{ambrosio2023functions, IglWal22} and vector-valued cases \cite{bredies2024extremal}, such questions remain essentially unexplored for the respective nonlocal variants. This paper aims to fill this gap by analyzing the geometry of the nonlocal functionals described above, highlighting the role of decomposability for the various notions of sets of finite nonlocal perimeter. We summarize here the results obtained in this paper. 

\textbf{Roadmap of the paper and main results.} In Section \ref{sec:nonlocalper} we start by considering Gagliardo-Slobodeckij-type spaces $W^{\rho,1}(\R^d)$ defined for a general kernel $\rho$ and their associated Gagliardo perimeter $\GPer_\rho(E) = |\1_E|_{W^{\rho,1}(\R^d)}$. In particular,  we focus on finite horizon settings, where the kernel $\rho(\cdot) = \rho_\varepsilon(\cdot)$ is supported on a ball of radius $\varepsilon$ and the associated Gagliardo $\GPer_\varepsilon$-perimeter. We demonstrate that a natural concept of $\varepsilon$-decomposability emerges, which can be defined through partitions with measure-theoretic distances greater or equal than $\varepsilon$. Building on this framework, we generalize the decomposition theorem established in \cite[Theorem 1]{AmbCasMasMor01} in the case of the sets of finite perimeters (see also \cite[4.2.25]{federer2014geometric} and \cite{kirchheim1998lipschitz} for similar statements). Specifically, we prove that any set with finite 
$\GPer_\varepsilon$-perimeter can be uniquely decomposed into a countable collection of its $\varepsilon$-connected components, that can be understood as its maximal $\varepsilon$-indecomposable subsets. Leveraging a recently established nonlocal isoperimetric inequality \cite{CesNov17}, we further provide a lower bound on the size of the 
$\varepsilon$-connected components of minimizers in variational problems involving the $\GPer_\varepsilon$-perimeter.
In Section \ref{sec:decfractional} we briefly show how the notion of indecomposability in the fractional case becomes meaningless. We prove that every set of finite fractional perimeter is indeed indecomposable, which can be seen in the previous framework as the case where $\varepsilon$ becomes $\infty$ (infinite horizon). Then, in Section \ref{sec:extepsfrac} we characterize the extreme points of the unit ball of the Gagliardo-Slobodeckij-type seminorm $|\cdot|_{W^{\rho_{\varepsilon}, 1}\left(\mathbb{R}^d\right)}$ and the fractional seminorm $|\cdot|_{W^{\alpha, 1}\left(\mathbb{R}^d\right)}$. In the former case, we prove that
\[
\Ext(\Bcal_{W^{\rho_\epsilon,1}})=\left\{ \pm \frac{\1_E}{\GPer_\epsilon(E)} \,\middle|\, \GPer_\epsilon(E) < \infty, \ E \ \text{is $\epsilon$-simple}\right\},
\]
where $\epsilon$-simple means that $E$ is $\epsilon$-indecomposable and $E^c$ has no $\epsilon$-connected component with finite measure. In the fractional case the extreme points are also indicator functions, however, in contrast to the extreme points of the classical TV ball \eqref{eq:extbvintro}, a notion of fractional indecomposability on $E$ is not required as pointed out above.

To round off the case of Gagliardo-type seminorms, in Section \ref{sec:gammaeps} we show how the notion of $\varepsilon$-indecomposability behaves when $\varepsilon \rightarrow 0$, i.e. when we localize the $\GPer_\varepsilon$-perimeter, and when $\varepsilon \rightarrow +\infty$, i.e. when we transition from finite horizons to infinite horizons. In order to derive a formal statement we consider $\GPer_\epsilon^c$, the so-called connected $\GPer_\varepsilon$-perimeter, as introduced in \cite{dayrens2022connected} for sets of finite perimeter. Under mild assumptions on the kernel $\rho_\varepsilon$ we prove that for a suitable dimensional constant $K_d$ it holds that 
    \[
    \Gamma\text{-}\lim_{\epsilon \to 0} \GPer_\epsilon^c = K_d\Per.
    \]
By choosing as $\rho_\varepsilon$ a truncated kernel that behaves asymptotically as the fractional kernel it also holds that 
   \[
    \Gamma\text{-}\lim_{\epsilon \to \infty} \GPer^c_\epsilon = \GPer_\alpha.
    \]
    
Turning to the distributional setting, in Section \ref{sec:distrBV} we extend the notion of fractional BV-functions from \eqref{eq:disBVintro} to more general radial kernels $\rho$ and study the associated function spaces $\BV^\rho(\R^d)$. If the kernel $\rho$ is compactly supported and under further suitable assumptions (see Section \ref{sec:distrBV}), we can exploit the recent results from \cite{BelMorSch24} to prove general properties for $\BV^\rho(\R^d)$, such as Poincar\'{e} inequalities, compactness results and the existence of an isomorphism $\mathcal{P}_\rho : \BV(\R^d) \rightarrow \BV^\rho(\R^d)$ that turns classical derivatives into their nonlocal counterpart. The corresponding notion of Caccioppoli perimeter $\Per_\rho$ and its decomposability is considered in Section \ref{sec:distrDec}, in which we are able to prove an analogous characterization of decomposability to the Gagliardo setting, see Theorem~\ref{thm:epsdecomp}, albeit with the additional assumption that the sets possess continuous unit normals. This is needed to ensure the normalized nonlocal gradient asymptotically aligns with the unit normal near the boundary. In Section \ref{sec:distrExt}, the extreme points of the unit ball $\Bcal_{\TV_\rho}$ are characterized, which are surprisingly not related to a nonlocal notion of decomposability, nor are they given by indicator functions. Instead, they are related to the extreme points of the unit ball in $\BV(\R^d)$ via the isomorphism $\Pcal_\rho$, that is,
\begin{align*}
    \Ext(\Bcal_{\TV_\rho})&=\left\{\Pcal_\rho\left(\frac{\pm \1_E}{\Per(E)}\right) \,\middle|\,   \Per(E)<\infty, \ E\ \text{is simple} \right\}\\
    &=\left\{u \in \Bcal_{\TV_\rho} \,\middle|\, D_\rho u = \pm\frac{\nu_E}{\Per(E)}\,\Hcal^{d-1}\mres{\partial^*E}, \ \Per(E)<\infty, \   E \ \text{is simple} \right\},
\end{align*}
with $\partial^\ast E$ the reduced boundary of $E$, and $\nu_E$ the generalized inner unit normal. While these extreme points are not indicator functions and will generally not have compact support, their nonlocal variation measure $D_\rho u$ is supported on the $(d-1)$-dimensional boundary of a simple set. A similar characterization also holds in the fractional case, where compact support of the kernel is missing. In fact, in that case the operator $\Pcal_\rho$ is simply the fractional Laplacian $(-\Delta)^{\frac{1-\alpha}{2}}$ and the extreme points can be exactly given in the one-dimensional case, cf.~Remark~\ref{rem:ext1d}\,(ii). 

Finally, Section \ref{sec:gammadistributional} complements the localization of the Gagliardo perimeter, by addressing the localization, in the sense of $\Gamma$-convergence, of the functionals $\TV_\rho$ and $\Per_\rho$ as the interaction range of the kernel vanishes. This extends the localization for nonlocal gradients in the Sobolev setting from \cite{cueto2024gamma} to the case of $\BV^\rho(\R^d)$; we also mention the results in \cite{MenSpe15}, which cover localization for functionals involving closely related nonlocal gradients but without an associated compactness statement, and the results in \cite{ComSte23} that contain localization results for the fractional variation and perimeter by instead letting the fractional parameter $\alpha$ tend to 1. As an application of our result, we are able to deduce that the minimizers of an isoperimetric problem involving $\Per_\rho$ must converge to balls as the interaction vanishes. This is noteworthy since it is currently unknown whether balls are the solution of the isoperimetric problem for the nonlocal or fractional Caccioppoli perimeter $\Per_\rho$ and $\Per_\alpha$. 
 
\textbf{Summary of the notations.} Due to the variety of different notions of nonlocal energies, we summarize here the main notations used in the paper.

\begin{itemize}
    \item $W^{\alpha,p}(\R^d)$: Gagliardo-Slobodeckij space for $1\leqslant p < \infty$ and $0<\alpha<1$.
    \item $W^{\rho,p}(\R^d)$: Gagliardo-Slobodeckij-type space defined for a general kernel $\rho$.
    \item $\TV_\rho$: Nonlocal total variation for a general kernel $\rho$.
    \item $\TV_\alpha$: $\TV_\rho$ when $\rho(\cdot) =  c_{d,\alpha}|\cdot|^{-d+1-\alpha}$, cf.~\eqref{eq:disBVintro}.
    \item $\BV^\rho(\R^d)$: Functions of nonlocal 
 bounded variations for a general kernel $\rho$.
    \item $\BV^\alpha(\R^d)$: $\BV^\rho(\R^d)$ when $\rho(\cdot) =  c_{d,\alpha}|\cdot|^{-d+1-\alpha}$.
    \item $\Per(E)$: Caccioppoli perimeter of $E \subset \R^d$.
    \item $\mathbb{P}_\rho(E)=|\1_E|_{W^{\rho,1}(\R^d)}$: Gagliardo perimeter of $E$ defined for a general kernel $\rho$.
    \item $\mathbb{P}_\varepsilon(E)$: $\mathbb{P}_\rho(E)$, when $\rho(\cdot) = \rho_\varepsilon(\cdot)$ is a general kernel supported in $\overline{B_\varepsilon(0)}$.
    \item $\mathbb{P}_\alpha(E)$: $\mathbb{P}_\rho(E)$ when $\rho(\cdot) = \rho_\alpha(\cdot)= |\cdot|^{-d+1-\alpha}$ is the fractional kernel, cf. \eqref{eq:fracperintro}.
    \item $\Per_\rho(E)$:  Nonlocal Caccioppoli perimeter of $E$ defined for a general kernel $\rho$.
    \item $\Per_\alpha(E)$: $\Per_\rho(E)$ when $\rho(\cdot) =  c_{d,\alpha}|\cdot|^{-d+1-\alpha}$, cf.~\eqref{eq:distrperintro}.
    \item $\Per_\varepsilon(E)$: $\Per_\rho(E)$ when $\rho(\cdot) = \rho_\varepsilon(\cdot)$ is a general kernel supported in $\overline{B_\varepsilon(0)}$.

\end{itemize}

\section{Extremality and decomposability with respect to Gagliardo perimeters}\label{sec:nonlocalper}
In this first section, we will consider Gagliardo-type seminorms with a general measurable interaction kernel $\rho : \R^d \rightarrow [0,\infty]$,
\begin{equation}
    |u|_{W^{\rho,1}(\R^d)}:= \int_{\R^d} \int_{\R^d} \frac{|u(x)-u(y)|}{|x-y|}\rho(x-y) \dd x \dd y
\end{equation}
and their associated Gagliardo perimeter $\GPer_\rho(E) = |\1_E|_{W^{\rho,1}(\R^d)}$. The following coarea formula (see \cite{Vis91}, \cite[Proposition 2.3]{CesNov17}) is valid for all $W^{\rho,1}(\R^d)$ seminorms:
\begin{equation}
\label{eq:Wrho1coarea}
|u|_{W^{\rho,1}(\R^d)} = \int_{-\infty}^{+\infty}  \GPer_{\rho} \big(\big\{ u > t \big\}\big) \dd t.
\end{equation}
We also remind the reader that by choosing $\rho(\cdot) = \rho_\alpha(\cdot)=|\cdot|^{-(d+\alpha-1)}$, one recovers the classical Gagliardo space 
$W^{\alpha,1}(\R^d)$ and the classical fractional perimeter $\GPer_\alpha(E)$, and refer to \cite{DinPalVal12, mazon2019nonlocal, Val13} for more details about the properties of such seminorms and perimeters.

We will mostly focus on the decomposability for interactions with a finite horizon, leaving the fractional Sobolev space case to the shorter Section \ref{sec:decfractional}. This decision is motivated by the fact that this is the case where nontrivial decomposability properties can be observed.  
To model finite horizon interactions we consider $\rho_{\varepsilon}$ with essential support given by $\overline{B_{\epsilon}(0)}$
and we define the corresponding $\GPer_\varepsilon$-perimeter as 
\begin{equation}
     \GPer_{\varepsilon} (E) := \big|\1_E\big|_{W^{\rho_{\varepsilon},1}(\R^d)}= 2\int_{E} \int_{E^c} \frac{1}{|x-y|}\rho_{\varepsilon}(x-y) \dd x \dd y.
\end{equation}
We will heavily use the so-called two-point gradient of a function $f:\R^d\to\R$, as introduced in \cite{DuGunLehZho13} for general interactions and similar to the notions used in \cite{MazSch18, AntDiaJinSch24} for the fractional case
\begin{equation}\label{eq:2pnonlocal}
    d_{\varepsilon} f:\R^d \times \R^d \to \R, \quad (d_{\varepsilon} f)(x,y):=\frac{f(x)-f(y)}{|x-y|}\rho_\varepsilon(x-y),
\end{equation}
\begin{equation}\label{eq:2pg}
    d_{\alpha} f:\R^d \times \R^d \to \R, \quad (d_{\alpha} f)(x,y):=\frac{f(x)-f(y)}{|x-y|^{d+\alpha}}.
\end{equation}

\subsection{Decomposability for finite horizon Gagliardo perimeters}\label{sec:finitehorizon}
A notion of decomposability for sets of finite Gagliardo perimeter can be given similarly to the classical notion of decomposability for Caccioppoli sets of finite perimeter \cite{AmbCasMasMor01}. Since we consider $\varepsilon$-horizon perimeters with kernels that are supported precisely in $\overline{B_\varepsilon(0)}$, our definition will depend on $\varepsilon$. Moreover, as will appear clear from Proposition \ref{prop:charepsilon}, such a notion depends only on $\varepsilon$ and not on the specific choice of the kernel $\rho_\varepsilon$.

\begin{defi}[$\varepsilon$-decomposability]\label{def:epsdec}
    We say that a set $E$ of finite $\GPer_\varepsilon$-perimeter is $\epsilon$-\emph{decomposable} if there exists a partition $\{E_1,E_2\}$ of $E$ into sets of finite $\GPer_\varepsilon$-perimeter such that $|E_1|,|E_2|>0$ and $ \GPer_{\varepsilon}(E)= \GPer_{\varepsilon}(E_1)+ \GPer_{\varepsilon}(E_2)$.
    On the contrary, we say that $E$ is $\epsilon$-\emph{indecomposable} if it is not $\epsilon$-decomposable.
\end{defi}

We now prove several properties of $\varepsilon$-indecomposable sets. We start with a characterization that will be fundamental for the rest of this section, linking the $\varepsilon$-decomposability of $E$ with the mutual distance of the supports of the Lebesgue measure restricted to the sets that partition $E$. We recall that for a nonnegative measure $\mu$ on $\R^d$, its support, denoted as $\supp \mu$,   is defined as the smallest closed set $F$ such that $\mu(\R^d \setminus F) = 0$, which directly implies
\[\supp \mu = \left\{ x \in \R^d \,\middle|\, \mu\big(B_{r}(x)\big) > 0 \text{ for all } r > 0\right\}.\]

\begin{prop}\label{prop:charepsilon}
    A set $E\subset \R^d$ is $\epsilon$-decomposable as in Definition $\ref{def:epsdec}$ if and only if there exists a partition $\{E_1,E_2\}$ of $E$ such that $|E_1|,|E_2|>0$ and
    \begin{equation}\label{eq:distance}
        \dist^e(E_1,E_2) := \inf _{\substack{x \in \supp(\L^d \mres E_1) \\ y \in \supp(\L^d \mres E_2)}}|x-y| \geqslant \varepsilon.
    \end{equation}
\end{prop}
\begin{proof}
    \textit{Sufficiency.} Consider finite $\GPer_\varepsilon$-perimeter sets $E_1,E_2 \subset E$ such that $E=E_1\cup E_2$ with $E_1\cap E_2=\emptyset$ and \eqref{eq:distance} holds. Since $E_1 \cap E_2=\emptyset$, we obtain using Tonelli's theorem that
    \begin{align}\label{eq:dec}
      \GPer_{\epsilon}(E_1) +  \GPer_{\epsilon}(E_2)=   \GPer_{\epsilon}(E)+ 4\int_{E_1\times E_2} \frac{1}{|x-y|}\rho_{\varepsilon}(x-y) \dd x \dd y.
\end{align}
Thanks to \eqref{eq:distance}, $\esssupp \rho_\epsilon = \overline{B_{\epsilon}(0)}$ and the fact that $|E_i\setminus \supp(\L^d \mres E_i)|=0$ for $i=1,2$, we immediately obtain that $  \GPer_{\epsilon}(E_1) +  \GPer_{\epsilon}(E_2)=   \GPer_{\epsilon}(E)$. This implies that $E_1$ and $E_2$ have finite $\GPer_\varepsilon$-perimeter and thus $E$ is $\varepsilon$-decomposable as in Definition \ref{def:epsdec}. 

\textit{Necessity.} Assume that $E$ is $\varepsilon$-decomposable as in Definition \ref{def:epsdec}, that is, there exists a partition $\{E_1,E_2\}$ of $E$ into sets of finite $\GPer_\varepsilon$-perimeter such that $|E_1|, |E_2| > 0$, and $\GPer_{\varepsilon}(E)= \GPer_{\varepsilon}(E_1)+ \GPer_{\varepsilon}(E_2)$. By \eqref{eq:dec}, it also holds that
\begin{equation}\label{eq:neczero}
    \int_{E_1\times E_2} \frac{1}{|x-y|}\rho_{\varepsilon}(x-y) \dd x \dd y=0.
\end{equation}
Suppose by contradiction that
\begin{equation}
        \inf _{\substack{x \in \supp(\L^d \mres E_1) \\ y \in \supp(\L^d \mres E_2)}}|x-y| < \varepsilon.
    \end{equation}
This implies that there exists a pair 
\begin{equation}
   (\bar x, \bar y)\in[\supp(\L^d \mres E_1)\times \supp(\L^d \mres E_2)]\cap [(E_1\times E_2)\cap \{|x-y|<\varepsilon\}]. 
\end{equation}
If we define $r:=\frac{\varepsilon-|\bar x-\bar y|}{2}$, then we obtain:
\begin{equation}\label{eq.1:prop2.4}
    [(E_1\cap B_{r}(\bar x))\times (E_2\cap B_{r}(\bar y))]\subset [(E_1\times E_2)\cap \{|x-y|<\varepsilon\}].
\end{equation}
Since $(\bar x, \bar y)\in \supp(\L^d \mres E_1)\times \supp(\L^d \mres E_2)$, we have that $|E_1\cap B_{r}(\bar x)|>0$ and $|E_2\cap B_{r}(\bar y)|>0$, which, thanks to \eqref{eq.1:prop2.4} and the essential support of $\rho_\epsilon$, immediately contradicts \eqref{eq:neczero}.
\end{proof}

\begin{rem}\label{rem:essinf}
    Note that 
    \begin{align*}
        \dist^e(E_1,E_2) = \essinf_{x \in E_1 , y \in  E_2}|x-y|,
    \end{align*}
    where the essential infimum is meant with respect to $\mathcal{L}^d \times \mathcal{L}^d$, as can be deduced immediately from the definition of $\supp(\L^d \mres E_i)$ and continuity of $z \mapsto |z|$.
\end{rem}
In the following, we denote by $E^{(1)}$ the measure theoretic interior of $E$, defined as

\begin{equation}
   E^{(1)}:=\left\{x \in \mathbb{R}^d: \lim _{r \rightarrow 0} \frac{\left|E \cap B_r(x)\right|}{\left|B_r(x)\right|}=1\right\},
\end{equation}
and by $E^{(0)}$ the measure theoretic exterior, defined as
\begin{equation}
   E^{(0)}:=\left\{x \in \mathbb{R}^d: \lim _{r \rightarrow 0} \frac{\left|E \cap B_r(x)\right|}{\left|B_r(x)\right|}=0\right\}.
\end{equation}
We emphasize that $E^{(0)}$ and $E^{(1)}$ satisfy
\begin{equation}\label{eq:densityeq}
\big|E \Delta E^{(1)}\big| = 0 \quad\text{and}\quad \big|E^c \Delta E^{(0)}\big| = 0,
\end{equation}
which in turn implies
\begin{equation}\label{eq:essdistint}\dist^e(E_1,E_2)=\dist\big(E_1^{(1)},E_2^{(1)}\big).\end{equation}
Moreover, we define the essential boundary as $\partial^e E = \R^d \setminus (E^{(0)} \cup E^{(1)})$, for which the following lemma holds.
\begin{lemma}\label{lem:essentialbdyspt}
Let $E$ be any measurable set. Then, we have that
\begin{equation}\label{eq:essentialbdyspt} \partial^e E \subset \supp(\L^d \mres E).\end{equation}
\end{lemma}
\begin{proof}
Let $x \in \partial^e E =  \R^d \setminus (E^{(0)} \cup E^{(1)})$. Then, there must be a sequence $r_n \searrow 0$ for which we have $|E \cap B_{r_n}(x)|/|B_{r_n}(x)| > 0$ for all $n \in \N$, since otherwise we would have $x \in E^{(0)}$. Noticing that $|E \cap B_r(x)|$ is monotone nonincreasing in $r$, this implies that $|E \cap B_r(x)| >0$ for all $r > 0$, that is, $x \in \supp(\L^d \mres E)$.
\end{proof}

As a consequence of the previous proposition, it holds that $\varepsilon$-decomposability implies decomposability with respect to the classical perimeter.  

\begin{prop}
Let $E$ be a set of finite (classical) Cacciopoli perimeter $\Per(E)$ and with $ \GPer_\epsilon(E) < \infty$. If $E$ is $\epsilon$-decomposable in the sense of Definition \ref{def:epsdec}, then it is also decomposable for the classical perimeter.
\end{prop}
\begin{proof}
By Proposition \ref{prop:frac_indec}, we have that $E = E_1 \cup E_2$ with $E_1,E_2$ satisfying \eqref{eq:distance}, so that by Lemma \ref{lem:essentialbdyspt} applied to $E_1$ and $E_2$ we also have 
\begin{equation}\label{eq:essbdydist}\inf_{\substack{x \in \partial^e E_1 \\ y \in \partial^e E_2}}|x-y| \geqslant \varepsilon.\end{equation}
Moreover, it holds that 
\begin{equation}\label{eq:essbdyunion}\partial^e E = \partial^e E_1 \cup \partial^e E_2.\end{equation}
Indeed, if $x \in \supp(\Lcal^d \mres E_1)$, then we have that $|E_2\cap B_r(x)|=0$ for all $r<\epsilon$ since $\dist^e(E_1,E_2)\geqslant \epsilon$. This implies that $x \not \in \partial^e E_2$ and $|E_1 \cap B_r(x)|=|E \cap B_r(x)|$ for all small $r$, which shows that $x \in \partial^e E$ if and only if $x \in \partial^e E_1 \cup \partial^e E_2$. On the other hand, if $x \not \in \supp(\Lcal^d \mres E_1)$, then $|E_1\cap B_r(x)|=0$ for all $r$ small enough. Hence, we can argue in the same way that $x \in \partial^e E$ if and only if $x \in \partial^e E_1 \cup \partial^e E_2$.

By De Giorgi structure theorem \cite[Thm.~15.9]{Mag12} and Federer theorem \cite[Thm.~16.2]{Mag12}, we have, for any finite perimeter set $F$, that
\begin{equation}\label{eq:perbdy}\partial^\ast F \subset \partial^e F, \quad \Ha^{d-1}\big( \partial^e F \setminus \partial^\ast F \big) = 0, \quad \text{and}\quad \Per(F) = \Ha^{d-1}(\partial^\ast F),\end{equation}
where $\partial^\ast F$ is the reduced boundary of $F$. Together with \eqref{eq:essbdyunion} and the assumption $\Per(E) < \infty$, this implies that $\Per(E_1), \Per(E_2) < \infty$. Now, using \cite[Thm.~16.3]{Mag12}, we can write
\begin{align*}\Per(E_1 \cup E_2) &= \Ha^{d-1}\left(\partial^\ast E_1 \cap E_2^{(0)}\right) + \Ha^{d-1}\left(\partial^\ast E_2 \cap E_1^{(0)}\right) \\&\quad + \Ha^{d-1}\big(\big\{x \in \partial^\ast E_1 \cap \partial^\ast E_2 :\, \nu_{E_1}(x) = \nu_{E_2}(x)\big\}\big),\end{align*}
in which the third term vanishes by \eqref{eq:essbdydist} and \eqref{eq:perbdy}. We are left to prove that 
\[\Ha^{d-1}\left(\partial^\ast E_1 \setminus E_2^{(0)}\right) = 0\quad\text{and}\quad \Ha^{d-1}\left(\partial^\ast E_2 \setminus E_1^{(0)}\right) = 0,\]
which (using again \eqref{eq:essbdydist} and \eqref{eq:perbdy}) is equivalent to 
\begin{equation}\label{eq:noperinside}\Ha^{d-1}\left(\partial^\ast E_1 \cap E_2^{(1)}\right) = 0\quad\text{and}\quad \Ha^{d-1}\left(\partial^\ast E_2 \cap E_1^{(1)}\right) = 0.\end{equation}
By an argument similar to the proof of Lemma \ref{lem:essentialbdyspt}, we have that, if $x \in E_i^{(1)}$, then $|E_i \cap B_r(x)| >0$ for all $r > 0$, so that $E_i^{(1)} \subset \supp(\L^d \mres E_i)$. This, together with \[\partial^\ast E_i \subset \partial^e E_i \subset \supp(\L^d \mres E_i)\] and \eqref{eq:distance}, implies \eqref{eq:noperinside}.
\end{proof}

\subsection{\texorpdfstring{$\varepsilon$}{eps}-connected components and isoperimetric inequality}
Inspired by the concept of $M$-connected components introduced in \cite{AmbCasMasMor01}, the first goal of this section is to develop a theory of $\varepsilon$-connected components of a set $E$. This will allow us to prove an analogous $\varepsilon$-decomposition theorem to \cite[Theorem 1]{AmbCasMasMor01} in our setting. 

We first introduce the notion of \emph{$\varepsilon$-connected component} of $E$ containing a point $x\in E^{(1)}$.
\begin{defi}[$\varepsilon$-connected component]
    Let $E \subset \R^d$, then we say that $x,y \in E^{(1)}$ are $\epsilon$-connected in $E$ if there are points $z_0,\ldots,z_N \in E^{(1)}$ with $z_0=x$, $z_N=y$ and 
    \[
    \dist(z_i,z_{i+1}) < \epsilon \quad \text{for all $i=0,\ldots,N-1$}.
    \]
    The $\varepsilon$-connected component $E^x$ of $E$ containing $x \in E^{(1)}$ is defined as 
    \begin{equation}\label{eq:defcomp}
E^x:=\left\{y \in E^{(1)} \,\middle|\, y \ \text{is $\epsilon$-connected to $x$ in $E$}\right\}.
\end{equation}
\end{defi}
\begin{lemma}\label{le:connind}
Let $E \subset \R^d$ be a finite $\GPer_\varepsilon$-perimeter set and $x,y \in E^{(1)}$. Then the following holds:
\begin{itemize}
    \item[(i)] $B_\epsilon(x) \cap E^{(1)} \subset E^x$.
    \item[(ii)] Either $E^x=E^y$ or $\dist(E^x,E^y)\geq \epsilon$.
    \item[(iii)] $\GPer_\epsilon(E^x)<\infty$ and $E^x$ is $\epsilon$-indecomposable.
\end{itemize}
\end{lemma}
\begin{proof}
(i) This is clear by definition, since every point in $B_\epsilon(x) \cap E^{(1)}$ is $\epsilon$-connected to $x$.

(ii) Suppose first that $x$ and $y$ are $\epsilon$-connected in $E$. Then, it is clear by concatenating sequences that $E^x=E^y$. Suppose now for the sake of contradiction that $x$ and $y$ are not $\epsilon$-connected, but $\dist(E^x,E^y)<\epsilon$. Then, there are $z \in E^x$ and $w \in E^y$ with $\dist(z,w)<\epsilon$. Therefore, given that $z$ is $\epsilon$-connected to $x$, we deduce that $w$ is $\epsilon$-connected to $x$ as well. Now using that $w$ is $\epsilon$-connected to $y$, we deduce that $x$ is $\epsilon$-connected to $y$, which is a contradiction.

(iii) We can compute that
\begin{align}
\begin{split}\label{eq:perex}
\GPer_\epsilon(E^x)&=2\int_{E^x}\int_{(E^x)^c}\frac{1}{|y-z|}\rho_\epsilon(y-z)\dd z\dd y\\
&=2\int_{E^x}\int_{E^c}\frac{1}{|y-z|}\rho_\epsilon(y-z)\dd z\dd y +2\int_{E^x}\int_{E \setminus E^x}\frac{1}{|y-z|}\rho_\epsilon(y-z)\dd z\dd y \\
&\leq \GPer_\epsilon(E)+2\int_{E^x}\int_{E \setminus E^x}\frac{1}{|y-z|}\rho_\epsilon(y-z)\dd z\dd y = \GPer_\epsilon(E),
\end{split}
\end{align}
where we have used that the second integral is zero given that $\dist^e(E^x, E \setminus E^x)\geq \epsilon$. This last observation follows from part (ii) together with the fact that $E^{(1)}$ is equal to the union of all $\epsilon$-connected components due to (i).
For the decomposability, assume to the contrary that we can find a partition $\{F_1,F_2\}$ of $E^x$ with $|F_1|,|F_2|>0$ and $\dist^e(F_1,F_2) \geqslant \varepsilon$ according to Proposition~\ref{prop:charepsilon}. Then, due to part (i) we may assume without loss of generality that $|B_r(x) \cap F_1|>0$ for all $r>0$. Note that this immediately implies $|B_\epsilon(x) \cap F_2|=0$, since otherwise $\dist^e(F_1,F_2)<\epsilon$. Next, take any $y \in E^x$ and consider a sequence $z_0,\ldots,z_N \in E^x$ with $z_0=x$, $z_N=y$ and such that $\dist(z_i,z_{i+1})<\epsilon$ for all $i=0,\ldots N-1$. If $|B_r(x_1)\cap F_2|>0$ for all $r>0$, then we find that $\dist^e(F_1,F_2)<\epsilon$. Hence, we must have $|B_r(x_1)\cap F_2|=0$ for all $r$ small enough, which implies $|B_r(x_1) \cap F_1|>0$ for all $r>0$. Continuing like this iteratively, we infer that $|B_r(y) \cap F_1|>0$ for all $r>0$ and, as a consequence, $|B_\epsilon(y)\cap F_2|=0$. Since $y \in E^x$ was arbitrary, we deduce that $|F_2|=0$, which contradicts the assumption.
\end{proof}
\begin{rem}
\leavevmode
\begin{itemize}
    \item [(i)] In the case that $E$ is $\epsilon$-indecomposable, then for any $x \in E^{(1)}$, we have $E^x=E^{(1)}$, that is, $E$ consists of a single $\epsilon$-connected component up to a null set. Indeed, if $|E \setminus E^x|>0$, then $\{E^x, E \setminus E^x\}$ would be a nontrivial partition of $E$ satisfying $\dist^e(E^x,E \setminus E^x) \geq \epsilon$ by Lemma~\ref{le:connind}\,(ii). This would contradict the $\epsilon$-decomposability of $E$ in light of Proposition~\ref{prop:charepsilon}.
    \item [(ii)]  If $E_1,E_2$ are two $\epsilon$-indecomposable sets with $\dist^e(E_1,E_2)<\epsilon$, then $E:=E_1 \cup E_2$ is $\epsilon$-indecomposable. Indeed, if we take $x \in E_1^{(1)}$ and $y \in E_2^{(1)}$, then the first part of this remark shows that $E_1^{(1)} \subset E^x$ and $E^{(1)}_2 \subset E^y$. But this implies that $\dist(E^x,E^y)<\epsilon$, so by Lemma~\ref{le:connind}\,(ii) we must have $E^x=E^y = E^{(1)}$, which yields that $E$ is $\epsilon$-indecomposable by Lemma~\ref{le:connind}\,(iii).\qedhere
\end{itemize}
\end{rem}

We can now prove the main $\epsilon$-decomposition result for sets with finite $\GPer_{\varepsilon}$-perimeter.
\begin{thm}[$\varepsilon$-decomposition theorem]\label{thm:epsdec}
    Let $E$ be a set with finite $\GPer_{\varepsilon}$-perimeter. Then, there exists a unique countable collection $\left\{E^{x_i}\right\}_{i \in I}$ of $\varepsilon$-connected components of $E$ such that $E^{x_i}$ is $\varepsilon$-indecomposable, $\left|E^{x_i}\right|>0$, $\bigcup_{i \in I}E^{x_i}=E^{(1)}$ and
    \begin{equation}\label{eq:epsdecom}
        \GPer_{\varepsilon}(E)=\sum_{i \in I} \GPer_{\varepsilon}(E^{x_i}).
    \end{equation}
\end{thm}
\begin{proof}
Note that by Lemma~\ref{le:connind}\,(i) and (ii), the $\epsilon$-connected components form a partition of $E^{(1)}$, that is, there exist $x_i \in E^{(1)}$ for $i \in I$ such that $E^{(1)}=\bigcup_{i \in I}E^{x_i}$ and $E^{x_i} \not = E^{x_j}$ for $i \not =j$. The uniqueness of this decomposition is immediate by Lemma~\ref{le:connind}\,(ii). Moreover, since $B_{\epsilon}(x_i)\cap E^{(1)} \subset E^{x_i}$ and all the distinct $\epsilon$-connected component are disjoint, the set $I$ can be at most countable. Finally, it remains to show \eqref{eq:epsdecom}. We compute as in \eqref{eq:perex}
\begin{align*}
    \sum_{i \in I} \GPer_{\varepsilon}(E^{x_i}) &= \sum_{i \in I}2\int_{E^{x_i}}\int_{E^c}\frac{1}{|y-z|}\rho_\epsilon(y-z)\dd z\dd y \\
    &\qquad+\sum_{I \in I}2\int_{E^{x_i}}\int_{E \setminus E^{x_i}}\frac{1}{|y-z|}\rho_\epsilon(y-z)\dd z\dd y\\
    &=\sum_{i \in I}2\int_{E^{x_i}}\int_{E^c}\frac{1}{|y-z|}\rho_\epsilon(y-z)\dd z\dd y\\
    &=2\int_{E^{(1)}}\int_{E^c}\frac{1}{|y-z|}\rho_\epsilon(y-z)\dd z\dd y=\GPer_\epsilon(E),
\end{align*}
where we have used that $\dist^e(E^{x_i},E\setminus E^{x_i}) \geq \epsilon$ in light of Lemma~\ref{le:connind}\,(ii).
\end{proof}

\begin{rem}
Since $|E^{(1)}\Delta E|=0$, we have that $\bigcup_{x \in E^{(1)}}E^x$ equals $E$ up to a null set.
Moreover, due to Lemma~\ref{le:connind}\,(ii) the $\epsilon$-connected components need to be at least $\varepsilon$ far apart from each other. This behavior is illustrated in Figure \ref{fig:epscomps}.
\end{rem}

\begin{figure}[ht]
    \centering
    \begin{tikzpicture}[scale=1.2]               
        \draw[color=OliveGreen, fill=OliveGreen!40, very thick] (0,0.866) circle (0.4);
        \draw[color=OliveGreen, fill=OliveGreen!40, very thick] (-0.5,0) circle (0.4);
        \draw[color=OliveGreen, fill=OliveGreen!40, very thick] (0.5,0) circle (0.4);
        \draw[color=OliveGreen, thick, densely dotted] ([shift=(-17:0.7)]0,0.866) arc (-18:198:0.7);
        \draw[color=OliveGreen, thick, densely dotted] ([shift=(107:0.7)]-0.5,0) arc (107:317:0.7);
        \draw[color=OliveGreen, thick, densely dotted] ([shift=(77:0.7)]0.5,0) arc (77:-138:0.7);
        \draw[color=OliveGreen] (-1.1,0.75) node {$E^{x_1}$}; 
        \draw (0.0,0.65) node {$x_1$};
        \filldraw (0,0.866) circle[radius=1.5pt];
        \draw[semithick] (1.3,0.5) -- (1.6,0.5);
        \draw[semithick] (1.3,0.45) -- (1.3,0.55);
        \draw[semithick] (1.6,0.45) -- (1.6,0.55);
        \draw (1.45,0.7) node {$\varepsilon$};
        \draw[color=RoyalBlue] (3.0,1.35) node {$E^{x_2}$};
        \draw[color=RoyalBlue, fill=RoyalBlue!40, very thick] (2.5,0.5) circle (0.4);
        \draw[color=RoyalBlue, fill=RoyalBlue!40, very thick] (3.5,0.5) circle (0.4);
        \draw[color=RoyalBlue, thick, densely dotted] ([shift=(-135:0.7)]3.5,0.5) arc (-135:135:0.7);
        \draw[color=RoyalBlue, thick, densely dotted] ([shift=(45:0.7)]2.5,0.5) arc (45:315:0.7);
        \draw (2.5,0.284) node {$x_2$};
        \filldraw (2.5,0.5) circle[radius=1.5pt];
        \draw[color=Maroon, fill=Maroon!40, very thick] (4.7,0.15) circle (0.4);
        \draw[color=Maroon, thick, densely dotted] (4.7,0.15) circle (0.7);
        \draw (4.7,-0.066) node {$x_3$};
        \filldraw (4.7,0.15) circle[radius=1.5pt];
        \draw[color=Maroon] (5.5,0.75) node {$E^{x_3}$}; 
    \end{tikzpicture}
    \caption{Decomposition of a set in three different $\varepsilon$-connected components. The dotted lines depict the boundaries of the sets $\{y\,|\, \dist(y, E^{x_i}) < \varepsilon\}$.}\label{fig:epscomps}
\end{figure}
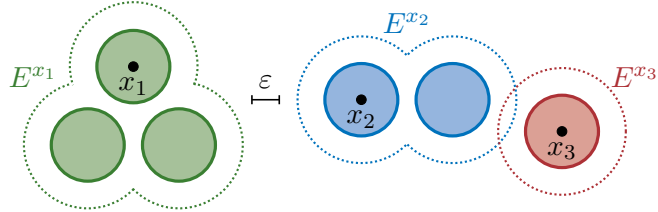

The aim of the next part of the section is to provide a lower bound for the measure of the $\varepsilon$-connected components for minimizers of 
\begin{equation}
    E \mapsto \GPer_{\varepsilon}(E)+\int_E g(x)\, \dd x,
\end{equation}
where $g \in L^{d / \alpha}(\R^d)$ and $\rho_{\varepsilon}(x-y)=\1_{[0, \varepsilon]}(|x-y|) |x-y|^{-(d+\alpha-1)}$. To this end, we need to recall several results on isoperimetric inequalities for Gagliardo perimeters in 
\cite[Proposition 3.1]{CesNov18} and \cite[Theorem 2.4]{mazon2019nonlocal}. 
\begin{thm}[Isoperimetric inequality]\label{thm:isop_ineq}
 For every measurable set $E \subset \mathbb{R}^d$ with finite measure it holds that
\begin{equation}\label{eq:isop_ineq}
    \GPer_{\varepsilon}(E) \geqslant \GPer_{\rho_{\varepsilon}^*}\left(B_{|E|}\right),
\end{equation}
where $\GPer_{\rho_\varepsilon^*}(F) := \int_F \int_{F^c} \frac{\rho^*_\varepsilon(x-y)}{|x-y|}\, \dd x \dd y$, $\rho_{\varepsilon}^*$ is the symmetric decreasing rearrangement of $\rho_{\varepsilon}$, and $B_{|E|}$ is the ball centered in zero such that $\left|B_{|E|}\right|=|E|$. 

In particular, if $\rho_{\varepsilon}$ is a radially nonincreasing function, then
\begin{equation}
  \GPer_{\varepsilon}(E) \geqslant \GPer_{\varepsilon}\left(B_{|E|}\right), 
\end{equation}  
where equality holds if and only if $E$ is a translated set of $B_{|E|}$.
\end{thm}

In particular, following \cite[Lemma 3.2]{CesNov18}, if we consider a truncated fractional kernel such as $\rho_{\varepsilon}(r)=\1_{[0, \varepsilon]}(r) r^{-(d+\alpha-1)}$, in the limit of the isoperimetric profile we recover the fractional power $|E|^{(d-\alpha)/d}$.
\begin{lemma}\label{lem:profilelimit}
    Let $\rho_{\varepsilon}(x-y)=\1_{[0, \varepsilon]}(|x-y|) |x-y|^{-(d+\alpha-1)}$. Then, it holds that
    \begin{equation}
    \lim _{|E| \rightarrow 0^{+}} \frac{\GPer_{\varepsilon}\left(B_{|E|}\right)}{|E|^{(d-\alpha)/d}}=c(d,\alpha) \in (0,+\infty).
\end{equation}
\end{lemma}
\begin{proof}
    Set $r = C |E|^{1/d}$ where $C = |B_1(0)|^{-1/d}$.
    Thanks to the change of variables $\bar x=x/r$ and $\bar y=y/r$, we can rewrite the $\GPer_\varepsilon$-perimeter as
    \begin{equation*}
        \GPer_{\varepsilon} (B_{|E|}) = \int_{B_{|E|}} \int_{B_{|E|}^c} \frac{\rho_{\varepsilon}(x-y)}{|x-y|} \dd x \dd y=\int_{B_{1}(0)} \int_{B_{1}(0)^c} \frac{\rho_{\varepsilon}^{r}(\bar x-\bar y)}{|\bar x-\bar y|}\dd \bar x \dd \bar y,
    \end{equation*}
    where $\rho_{\varepsilon}^{r}(x-y)=r^{2d-1} \rho_{\varepsilon}\big(r(x-y)\big)$. In particular, we have:
    \begin{equation*}
        \frac{\rho_{\varepsilon}^{r}(x-y)}{| x- y|}=\frac{r^{d-\alpha}\1_{[0, \frac{\varepsilon}{r}]}(| x- y|)}{| x- y|^{d+\alpha}}=C^{d-\alpha}|E|^{\frac{d-\alpha}{d}}\frac{\1_{[0, \frac{\varepsilon}{r}]}(| x- y|)}{|x- y|^{d+\alpha}}
    \end{equation*}
    for almost every $x,y \in B_1(0)$.    Therefore, as $r\rightarrow 0$
    \begin{equation*}
    \frac{\rho_{\varepsilon}^{r}(x-y)}{r^{d-\alpha}}\rightarrow \frac{1}{|x-y|^{d+\alpha-1}},
    \end{equation*}
    which gives
    \begin{equation*}
        \lim _{|E| \rightarrow 0^{+}} \frac{\GPer_{\varepsilon}\left(B_{|E|}\right)}{|E|^{(d-\alpha)/d}}=C^{d-\alpha} \int_{B_{1}(0)} \int_{B_{1}(0)^c} \frac{1}{|x-y|^{d+\alpha}}\dd x \dd  y=:c(d,\alpha).\qedhere
    \end{equation*}
\end{proof}

\begin{prop}Let $\rho_\varepsilon$ be as in Lemma \ref{lem:profilelimit}, $g \in L^{d/\alpha}(\R^d)$,
\begin{equation}\label{eq:capillarity}E \in \argmin_F \ \GPer_\varepsilon(F) + \int_F g(x) \dd x,\end{equation}
and $E^{x_i}$ be the $\varepsilon$-connected components of $E$ provided by Theorem \ref{thm:epsdec}. Then, there is a constant $c(\alpha,d,g) >0$ such that 
\begin{equation}\label{eq:masslowerbound}|E^{x_i}| > c(\alpha, d, g)\quad\text{for all }i \in I.\end{equation}
\end{prop}
\begin{proof}
For each $\epsilon$-connected component $E^{x_i}$ we have that 
\begin{equation}\label{eq:negative}
\GPer_\epsilon(E^{x_i}) + \int_{E^{x_i}} g(x)\dd x \leq 0,
\end{equation}
since otherwise we can remove the component $E^{x_i}$ from $E$ to reduce the value of the functional. On the other hand, by Theorem \ref{thm:isop_ineq} and Lemma \ref{lem:profilelimit}, there are $m_0 > 0$ and $\tilde{c}(d,\alpha)$ such that 
\[\GPer_\varepsilon(F) \geqslant \tilde{c}(d,\alpha) |F|^{(d-\alpha)/d} \quad \text{ whenever }|F| \leqslant m_0,\]
so that using \eqref{eq:negative} and H\"older's inequality, we get
\[\tilde{c}(d,\alpha) |E^{x_i}|^{(d-\alpha)/d} \leqslant \GPer_\varepsilon(E^{x_i}) \leqslant \left|\int_{E^{x_i}} g(x) \dd x \right| \leqslant \|g\|_{L^{d/\alpha}(E^{x_i})} |E^{x_i}|^{(d-\alpha)/d} \]
if $|E^{x_i}| \leqslant m_0$. Since we have assumed $g \in L^{d/\alpha}(\R^d)$, we have that for each $\delta >0$ there is $m>0$ so that $|F| \leq m$ implies $\|g\|_{L^{d/\alpha}(F)} < \delta$. Applying this with $\delta = \tilde{c}(d,\alpha)$ provides the desired bound \eqref{eq:masslowerbound} with $c(\alpha, d, g) = \min(m, m_0)$, since otherwise the estimate above would lead to a contradiction.
\end{proof}

\subsection{Decomposability for the fractional perimeter}\label{sec:decfractional}

Due to the infinite horizon of the fractional kernel, we show here that the notion of indecomposability for sets of finite fractional perimeter is not meaningful, since all sets with finite fractional perimeter are indecomposable with respect to the fractional perimeter $\GPer_\alpha$.

\begin{prop}\label{prop:frac_indec}
Let $E_1$ and $E_2$ be sets of finite fractional perimeter such that $E = E_1 \cup E_2$ with $|E_1 \cap E_2| = 0$ and $|E_1|,|E_2| > 0$, then
\begin{equation}\label{eq:fracperineq} \GPer_{\alpha}(E) =  \GPer_{\alpha}(E_1 \cup E_2) <  \GPer_{\alpha}(E_1) +  \GPer_{\alpha}(E_2).\end{equation}
\end{prop}
\begin{proof}
Since $|E_1 \cap E_2| = 0$, we can compute as in \eqref{eq:dec} that
\begin{align}\label{eq:estind}
      \GPer_{\alpha}(E_1) +  \GPer_{\alpha}(E_2)=   \GPer_{\alpha}(E)+ 4\int_{E_1\times E_2} \frac{1}{|x-y|^{d+\alpha}} \dd x \dd y.
\end{align}
Hence, if $|E_1|,|E_2| > 0$, then \eqref{eq:fracperineq} follows from \eqref{eq:estind}.
\end{proof}

\begin{rem}The only property of $\rho_\alpha$ that was used in Proposition \ref{prop:frac_indec} is that $\rho_\alpha(x) > 0$ for all $x \in \R^d$. Roughly speaking, we can also see this result as the case of $\varepsilon = + \infty$ in Proposition \ref{prop:charepsilon}.
\end{rem}

\subsection{Extreme points in \texorpdfstring{$W^{\rho_\varepsilon,1}(\R^d)$}{W1rhoeps} and \texorpdfstring{$W^{\alpha,1}(\R^d)$}{W1alpha}}\label{sec:extepsfrac}

In this section, we aim to characterize the extreme points of the Gagliardo-type seminorm both in the finite horizon case and in the fractional case. The notion of decomposability for sets of finite Gagliardo perimeter is directly linked to the characterization of extreme points for the Gagliardo-type seminorm. This aligns with the case of sets of finite perimeter, where the extreme points of the classical total variation ball are normalized indicator functions of simple sets (indecomposable and saturated sets) \cite{AmbCasMasMor01}.  

We start with a simple lemma.

\begin{lemma}\label{lem:epspermorethanzero}
Let $E \subset \R^d$ be such that $|E| > 0$ and $|E^c|>0$. Then, $\GPer_\varepsilon(E)>0$.
\end{lemma}
\begin{proof}
We proceed by contradiction. We assume that $\GPer_{\varepsilon} (E) = 0$, which can be rewritten using the Tonelli theorem as
\[\GPer_{\varepsilon} (E) = 2\int_{E} \int_{E^c} \frac{1}{|x-y|}\rho_{\varepsilon}(x-y) \dd x \dd y  = 0.\]
Since the integrand is nonnegative and $\esssupp \rho_\varepsilon=\overline{B_{\epsilon}(0)}$, this is equivalent to
\[|x-y| \geqslant \varepsilon \quad \text{ for a.e. } (x,y) \in E^c \times E,\]
or, equivalently,
\[
\dist^e(E,E^c) = \dist( \supp (\Lcal^d \mres E), \supp(\Lcal^d \mres E^c)) \geqslant \epsilon.
\]
However, it can be directly verified from the definition that $\supp (\Lcal^d \mres E) \cup \supp(\Lcal^d \mres E^c) = \R^d$, so that $\dist^e(E,E^c)=0$. This yields the desired contradiction.
\end{proof}

We are now ready to characterize the extreme points of 

\[\Bcal_{W^{\rho_{\varepsilon},1}}:=\left\{u \in W^{\rho_{\varepsilon},1}(\R^d)\,\middle|\,|u|_{W^{\rho_{\varepsilon},1}(\R^d)} \leqslant1\right\}.\]
We introduce the following notion first, mirroring the classical concept of simple sets.
\begin{defi}[$\epsilon$-simple]
    A set $E\subset \R^d$ with $|E|<\infty$ is called $\epsilon$-simple, if $E$ is $\epsilon$-indecomposable and $E^c$ does not have any $\epsilon$-connected component with finite measure.
\end{defi}
\begin{thm}\label{thm:exteps}
  The extreme points of $\Bcal_{W^{\rho_{\varepsilon},1}}$ can be characterized as
  \begin{align}
\Ext(\Bcal_{W^{\rho_{\varepsilon},1}}) = \Bigg\{ \pm\frac{\1_E}{ \GPer_{\epsilon}(E)} \,\Bigg|\, & E \subset \R^d \text{ with } |E|\in (0,\infty), \ \GPer_{\epsilon}(E) \in (0,\infty) \text{ and }E \text{ $\epsilon$-simple}\Bigg\}\label{eq:epsilonext}.
  \end{align}
\end{thm}
\begin{proof}
\emph{Sufficiency.} Let $u = \frac{1}{ \GPer_{\epsilon}(E)}\1_E$ with $|E| \in (0,\infty)$, $\GPer_\epsilon(E) \in (0,\infty)$ and such that $E$ is $\epsilon$-simple (the proof is similar for $u=-\1_E/\GPer_{\varepsilon}(E)$). We aim to prove that $u \in \Ext(\Bcal_{W^{\rho_{\varepsilon},1}})$. 
Suppose that there exist two functions $u_1,u_2 \in \mathcal{B}_{W^{\rho_\varepsilon,1}}$ and $\lambda \in (0,1)$ such that 
\begin{align}\label{eq:decom}
    u = \lambda u_1 + (1-\lambda) u_2.
\end{align}
Note that since $|u|_{W^{\rho_\varepsilon,1}}=1$, we also have $|u_1|_{W^{\rho_\varepsilon,1}} = |u_2|_{W^{\rho_\varepsilon,1}}=1$.
Then, applying the two-point gradient \eqref{eq:2pnonlocal} to the identity above, we get that $d_\varepsilon u = \lambda d_\varepsilon u_1 + (1-\lambda) d_\varepsilon u_2$. Since $d_\varepsilon u$ is essentially supported on $[E \times E^c] \cup [E^c \times E]$, we deduce that $d_\varepsilon u_1$ and $d_\varepsilon u_2$ are also essentially supported on this set. Indeed, if this were not the case, we would obtain that
 \begin{align}
     |u|_{W^{\rho_\varepsilon,1}(\R^d)} & = 2\int_{E} \int_{E^c} \frac{|u(x)-u(y)|}{|x-y|}\rho_\varepsilon
(x-y) \dd x \dd y\\
     &\leqslant 2\int_{E} \int_{E^c} \lambda\frac{|u_1(x)-u_1(y)|}{|x-y|}\rho_\varepsilon
(x-y) \dd x \dd y \\
& \qquad \qquad \qquad + 2\int_{E} \int_{E^c} (1-\lambda)\frac{|u_2(x)-u_2(y)|}{|x-y|}\rho_\varepsilon
(x-y) \dd x \dd y\\
     & < \lambda |u_1|_{W^{\rho_\varepsilon,1}(\R^d)}+(1-\lambda)|u_2|_{W^{\rho_\varepsilon,1}(\R^d)},
 \end{align}
which gives a contradiction since $1=|u|_{W^{\rho_\varepsilon,1}(\R^d)}=\lambda |u_1|_{W^{\rho_\varepsilon,1}(\R^d)}+(1-\lambda)|u_2|_{W^{\rho_\varepsilon,1}(\R^d)}$.  
This implies that $d_{\varepsilon} u_i=0$ a.e.~in $E\times E$ and $E^c \times E^c$ for $i=1,2$. 
Knowing that $E$ is $\varepsilon$-simple, we want to show that $u_i = c_i\mathds{1}_E$ for $c_i \in \R$. We only consider the case $i=1$ and first prove that $u_1 = c$ a.e.~in $E$. Suppose by contradiction that $u_1$ is not equal to a constant a.e.~in $E$. Then, there exists a constant $\eta \in \R$ such that 
\begin{align}
E_1 := \{x\in E: u_1(x) \geqslant\eta\} \quad \text{and} \quad E_2 := \{x\in E: u_1(x) < \eta\} 
\end{align}
are sets of positive measure. In particular, $|u_1(x) - u_1(y)| >0$ for a.e.~$(x,y) \in E_1 \times E_2$. Note that $\{E_1,E_2\}$ is a partition of $E$. Since $E$ is $\varepsilon$-indecomposable and due to Proposition \ref{prop:charepsilon}, it holds that there exist $\bar x \in \supp(\mathcal{L}^d \res E_1)$ and $\bar y \in \supp(\mathcal{L}^d \res E_2)$ such that $|\bar x-\bar y| < \varepsilon$. In particular, by arguing in a similar way as Proposition \ref{prop:charepsilon} and defining $r:= \frac{\varepsilon - |\bar x-\bar y|}{2}$, it holds that 
$|B_r(\bar x) \cap E_1| > 0$, $|B_r(\bar y) \cap E_2| > 0$ and 
\begin{align}
[(E_1\cap B_{r}(\bar x))\times (E_2\cap B_{r}(\bar y))]\subset [(E_1\times E_2)\cap \{|x-y|<\varepsilon\}].
\end{align}
This implies that $\left.d_{\varepsilon} u_1\right|_{E \times E}>0$ on a set of positive measures, leading to a contradiction. In the same way, one can argue that $u_1$ must be constant on each $\epsilon$-connected component of $E^c$. Since each of these connected components has infinite measure due to $E$ being $\epsilon$-simple, and since $u_1$ is integrable, we must have that $u_1=0$ a.e.~in $E^c$. We conclude that $u_i = c_i\mathds{1}_E$ for $i=1,2$. Finally, since $|u_i|_{W^{\rho_\varepsilon,1}}=1$, we deduce that $c_i = \frac{1}{\GPer_\varepsilon(E)}$, implying that $u_1 = u_2 =u$ and thus concluding the proof.

\emph{Necessity.} Given $u \in \Ext(\Bcal_{W^{\rho_{\varepsilon},1}})$, we want to prove that $u=\pm\frac{1}{ \GPer_{\epsilon}(E)}\1_E$ for some $E \subset \R^d$, with $|E| \in (0, \infty)$, $\GPer_\epsilon(E) \in (0,\infty)$  and $E$ $\varepsilon$-simple. We first note that $|u|_{W^{\rho_\epsilon,1}(\R^d)} =1$ must hold, since otherwise we can write $u$ as a convex combination by scaling. Furthermore, we assume without loss of generality that there is some $\eta>0$ such that $|\{u > \eta \}|>0$, otherwise we can replace $u$ by $-u$. Define the non-constant function $u_\eta:=\max\{u-\eta,0\} \in L^1(\R^d)$, then we find by the coarea formula \eqref{eq:Wrho1coarea} that
\[
|u_\eta|_{W^{\rho_\epsilon,1}(\R^d)} = \int_\eta^\infty \GPer_\epsilon \left(\left\{u > t\right\}\right)\dd t
\]
and
\[
\quad |u-u_\eta|_{W^{\rho_\epsilon,1}(\R^d)}=|\min\{u,\eta\}|_{W^{\rho_\epsilon,1}(\R^d)}= \int_{-\infty}^\eta \GPer_\epsilon \left(\left\{u > t\right\}\right)\dd t.
\]
Hence, we find that
\[
|u|_{W^{\rho_\epsilon,1}(\R^d)} = |u_\eta|_{W^{\rho_\epsilon,1}(\R^d)}+|u-u_\eta|_{W^{\rho_\epsilon,1}(\R^d)}.
\]
Observe that both $|u_\eta|_{W^{\rho_\epsilon,1}(\R^d)}$ and $|u-u_\eta|_{W^{\rho_\epsilon,1}(\R^d)}$ are positive, due to Lemma~\ref{lem:epspermorethanzero} and the coarea formula. Therefore, if we set $\lambda:=|u_\eta|_{W^{\rho_\epsilon,1}(\R^d)} \in (0,1)$ and define the functions
\[
u_1 = \frac{1}{\lambda}u_\eta \quad \text{and} \quad u_2 = \frac{1}{1-\lambda}(u-\lambda u_1),
\]
then $u_1,u_2 \in \Bcal_{W^{\rho_{\varepsilon},1}}$ and $u=\lambda u_1 + (1-\lambda)u_2$. Since $u$ is an extreme point, we must have $u=u_1$, or equivalently, $u=u_\eta/\lambda$. This latter condition is only possible if $u$ is constant on $\{u>\eta\}$ and zero outside this set, that is, with $E:=\{u > \eta \} \subset \R^d$ it holds that $u=\mathds{1}_E/\GPer_\epsilon(E)$.

It remains to prove that $E$ is $\epsilon$-simple. Suppose by contradiction that $E$ is $\varepsilon$-decomposable. Then, there exists a partition $\{E_1,E_2\}$ of $E$ in sets of finite $\GPer_{\varepsilon}$-perimeter such that $|E_1|, |E_2| \in (0,\infty)$ and $\GPer_\varepsilon(E) = \GPer_\varepsilon(E_1) + \GPer_\varepsilon(E_2)$. By Lemma \ref{lem:epspermorethanzero}, we have that $\GPer_\varepsilon(E_1) >0$ and $\GPer_\varepsilon(E_2) > 0$. Then, by defining 
\begin{align}
u_1 = \frac{\1_{E_1}}{\GPer_\varepsilon(E_1)} \quad \text{and} \quad u_2 = \frac{\1_{E_2}}{\GPer_\varepsilon(E_2)},
\end{align}
it holds that $u_1, u_2 \in \Bcal_{W^{\rho_\varepsilon,1}(\R^d)}$. Therefore, since $E_1 \cap E_2= \emptyset$ and $\GPer_\varepsilon(E) = \GPer_\varepsilon(E_1) + \GPer_\varepsilon(E_2)$, we obtain
\begin{align}
    u = \frac{\GPer_\varepsilon(E_1)}{\GPer_\varepsilon(E)} u_1 + \frac{\GPer_\varepsilon(E_2)}{\GPer_\varepsilon(E)}u_2,
\end{align}
contradicting the extremality of $u$. Suppose now by contradiction that $E^c$ has an $\epsilon$-connected component with finite measure. Let $\{F_1,F_2\}$ be an $\varepsilon$-decomposition of $E^c$ according to Definition~\ref{def:epsdec} with $|F_1|\in (0,\infty)$. Define 
\begin{align}
    u_1 := - \frac{\1_{F_1}}{\GPer_\varepsilon(F_1)} \quad \text{and} \quad u_2 := \frac{1 - \1_{F_2}}{\GPer_\varepsilon(F_2)},
\end{align}
which are integrable functions with $u_1,u_2 \in \Bcal_{W^{\rho_\epsilon,1}(\R^d)}$. Then, by similar considerations as before, it holds that 
\begin{align}
u =   \frac{\GPer_\varepsilon(F_1)}{\GPer_\varepsilon(E^c)} u_1 + \frac{\GPer_\varepsilon(F_2)}{\GPer_\varepsilon(E^c)}u_2, 
\end{align}
contradicting the extremality of $u$.
\end{proof}

\begin{rem}
The expression of $\GPer_\varepsilon$ as a double integral is not the only reasonable definition of a perimeter with radius of interaction $\varepsilon$. An alternative is the Minkowski-type perimeter studied in \cite{CesDipNovVal18, CesNov17} and defined (up to a factor $2$ with respect to their notation) as
\begin{equation}\label{eq:minkper}\widetilde{\Per}_{\varepsilon,1,1}(E):= \frac{1}{\varepsilon} \big(\big|E^{(1)} \oplus B_\varepsilon(0)\big| - \big|E^{(1)} \ominus B_\varepsilon(0)\big|\big),
\end{equation}
where for any $F \subset \R^d$
\[
F \oplus B_\varepsilon(0) := \left\{x \in \R^d \,\middle|\, \dist(x,F)<\epsilon\right\} \quad \text{and} \quad F \ominus B_\varepsilon(0) := \left\{x \in \R^d\,\middle|\, \dist(x,F^c)\geq \epsilon\right\}.
\]
Its generalization with respect to two nonnegative measures $\rho_0, \rho_1$ which are absolutely continuous with respect to the Lebesgue measure is given, taking into account \eqref{eq:densityeq}, by
\begin{equation}\label{eq:bunper}\begin{aligned}\widetilde{\Per}_{\varepsilon,\rho_0,\rho_1}(E):&= \frac{1}{\varepsilon} \,\rho_0\!\left( \left\{ x \in E^c \,\middle|\, \operatorname{dist}\big(x, E^{(1)}\big) < \varepsilon \right\}\right) + \frac{1}{\varepsilon}\, \rho_1\!\left( \left\{ x \in E \,\middle|\, \operatorname{dist}\big(x, E^{(0)}\big) < \varepsilon \right\}\right)\\
&=\frac{1}{\varepsilon} \,\rho_0\!\left( \left\{ x \in E^{(0)} \,\middle|\, \operatorname{dist}\big(x, E^{(1)}\big) < \varepsilon \right\}\right) + \frac{1}{\varepsilon}\, \rho_1\!\left( \left\{ x \in E^{(1)} \,\middle|\, \operatorname{dist}\big(x, E^{(0)}\big) < \varepsilon \right\}\right),
\end{aligned}\end{equation}
where for the second line we have used absolute continuity of $\rho_0, \rho_1$ and \eqref{eq:densityeq}. The latter was introduced in \cite{BunTriMur23} and further studied in \cite{BunKer24}, motivated by machine learning applications. Specifically, $\widetilde{\Per}_{\varepsilon,\rho_0,\rho_1}$ can be seen (up to a normalization of both terms) as the functional minimized by a binary classifier to distinguish between the two conditional data distributions $\rho_0,\rho_1$, which should be robust with respect to adversarial perturbations of the data up to a distance $\varepsilon$, referred to in this context as adversarial budget. 

Assuming that $\supp \rho_0 = \R^d$, it holds that a set $E$ is decomposable with respect to $\widetilde{\Per}_{\varepsilon,\rho_0,\rho_1}$ if and only if it is $2\varepsilon$-decomposable. Indeed, if $E=E_1 \cup E_2$ with $|E_1|,|E_2|>0$ and $E_1\cap E_2=\emptyset$, then $\widetilde{\Per}_{\varepsilon,\rho_0,\rho_1}(E) = \widetilde{\Per}_{\varepsilon,\rho_0,\rho_1}(E_1) + \widetilde{\Per}_{\varepsilon,\rho_0,\rho_1}(E_2)$ is equivalent to
\begin{equation}\label{eq:rhozero}
\rho_0\left( \left\{ x \in E^{(0)} \,\middle|\, \operatorname{dist}\big(x, E^{(1)}\big) < \varepsilon \right\}\right)=\sum_{i=1,2}\rho_0\left( \left\{ x \in E_i^{(0)} \,\middle|\, \operatorname{dist}\big(x, E_i^{(1)}\big) < \varepsilon \right\}\right)
\end{equation}
and
\begin{equation}\label{eq:rhoone}
\rho_1\left( \left\{ x \in E^{(1)} \,\middle|\, \operatorname{dist}\big(x, E^{(0)}\big) < \varepsilon \right\}\right)=\sum_{i=1,2}\rho_1\left( \left\{ x \in E_i^{(1)} \,\middle|\, \operatorname{dist}\big(x, E_i^{(0)}\big) < \varepsilon \right\}\right).
\end{equation}
Now if $\dist(E_1^{(1)},E_2^{(1)})\geq 2\epsilon$, then \eqref{eq:rhozero} and \eqref{eq:rhoone} hold due to the fact that the sets on the left hand side are the disjoint union of the sets on the right. On the other hand, if \eqref{eq:rhozero} holds, then due to $\supp \rho_0 = \R^d$ we must have that
\[
\left\{ x \in E_i^{(0)} \,\middle|\, \operatorname{dist}\big(x, E_i^{(1)}\big) < \varepsilon \right\} \cap E_j^{(1)} \quad \text{for $i \not =j$}
\]
and
\[
\left\{ x \in E_1^{(0)} \,\middle|\, \operatorname{dist}\big(x, E_1^{(1)}\big) < \varepsilon \right\} \cap \left\{ x \in E_2^{(0)} \,\middle|\, \operatorname{dist}\big(x, E_2^{(1)}\big) < \varepsilon \right\}
\]
are null sets. Since also $E_1^{(1)} \cap E_2^{(1)}$, $\partial^e E_1$ and $\partial^e E_2$ are null sets, we find by taking the union of all these sets that
\[
\left\{ x \in \R^d \,\middle|\, \operatorname{dist}\big(x, E_1^{(1)}\big) < \varepsilon \right\} \cap \left\{ x \in \R^d \,\middle|\, \operatorname{dist}\big(x, E_2^{(1)}\big) < \varepsilon \right\}
\]
has zero measure. This shows that $\dist(E_1^{(1)},E_2^{(1)})\geq 2\epsilon$, that is, $E$ is $2\epsilon$-decomposable.

Under the same assumption on $\rho_0$, a characterization of extreme points completely analogous to Theorem \ref{thm:exteps} follows for the unit ball of the total variation defined from $\widetilde{\Per}_{\varepsilon,\rho_0,\rho_1}$ through the coarea formula in \cite[Eq.~(4.2)]{BunKer24}.
\end{rem}

With similar techniques we can characterize the extreme points of the unit ball of the fractional seminorm 
\[\Bcal_{W^{\alpha,1}}:=\left\{u \in W^{\alpha,1}(\R^d)\,\middle|\,|u|_{W^{\alpha,1}(\R^d)} \leqslant1\right\}.\]
We note that in this case, as a consequence of Proposition \ref{prop:frac_indec}, the notion of indecomposability does not play a role. This is apparent also from Proposition \ref{prop:charepsilon} by taking $\epsilon=\infty$. The proof is omitted since the argument is identical to Theorem~\ref{thm:exteps} by replacing $\epsilon$ by $\infty$.

\begin{thm}\label{thm:extfrac}
The extreme points of $\Bcal_{W^{\alpha,1}}$ can be characterized as
\begin{equation}\Ext(\Bcal_{W^{\alpha,1}}) = \left\{ \pm\frac{\1_E}{ \GPer_{\alpha}(E)} \,\middle|\, E \subset \R^d, \, |E| \in (0,\infty) \ \text{and} \  \GPer_{\alpha}(E)\in (0,\infty)\right\}.
\label{eq:sobolevext}
\end{equation}
\end{thm}


\section{Relaxation and \texorpdfstring{$\Gamma$}{Gamma}-convergence for \texorpdfstring{$\GPer_\varepsilon$}{Peps} with \texorpdfstring{$\varepsilon$}{eps}-indecomposability constraints}\label{sec:gammaeps}
In this section, we consider the connected $\GPer_{\varepsilon}$-perimeter and its convergence to the classical perimeter as $\epsilon \to 0$. First, let $\Mcal_d$ denote the collection of Lebesgue measurable subsets of $\R^d$ endowed with the convergence in measure, that is, $E_n \to E$ if $|E_n \Delta E| \to 0$. We note that this definition does not exclude infinite measure sets. Furthermore, following \cite{dayrens2022connected} which shows that the classical perimeter under indecomposability constraints has a nontrivial relaxation when $d=2$, we introduce the connected $\GPer_\epsilon$-perimeter $\GPer_\epsilon^{c}:\Mcal_d \to [0,\infty]$ as
\[
\GPer_\epsilon^c(E):=\begin{cases}
    \GPer_\epsilon(E) &\text{if $E$ is $\epsilon$-indecomposable,}\\
    \infty &\text{else.}
\end{cases}
\]
We aim to show that $\GPer_\epsilon^c$ $\Gamma$-converges to the classical perimeter as $\epsilon \to 0$ for all $d$. To achieve this, we assume $\rho_\epsilon$ to be radial with  $\esssupp \rho_\epsilon = \overline{B_\varepsilon(0)}$ for all $\epsilon >0$ and
\begin{align}\label{eq:mean}
\int_{\R^d} \rho_\epsilon(h)\dd{h} = 1 \quad \text{for all $\epsilon>0$.}
\end{align}
We first state the following lemma, which gives a bound on $\GPer_\epsilon$ and shows that under assumption \eqref{eq:mean}, all finite perimeter sets also have finite $\GPer_\epsilon$-perimeter. The proof is omitted since it is a simple adaptation of \cite[Lemma~1 and 3]{Pon04}.
\begin{lemma}\label{le:perimeterbound}
For all $r>0$ and $E \in \Mcal_d$, it holds that
\[
\GPer_\epsilon(E) \leqslant K_d\Per(E)\int_{|h|<r} \rho_\epsilon(h)\dd{h}  + \frac{2|E|}{r}\int_{|h|>r}\rho_\epsilon(h)\dd{h},
\]
with $K_d:=|\S^{d-1}|^{-1}\int_{\S^{d-1}}|e_1\cdot \sigma|\dd{\Hcal^{d-1}(\sigma)}$. In particular, it holds that $\GPer_\epsilon \leqslant K_d \Per$.
\end{lemma}
As a consequence, we obtain that the $\GPer_\epsilon$-perimeter of a ball converges to zero as the radius vanishes, which is false for the classical perimeter when $d=1$.
\begin{cor}\label{cor:perimeterball}
    It holds that
    \[
    \lim_{\delta \to 0}\GPer_\epsilon(B_\delta(0))=0. 
    \]
\end{cor}
\begin{proof}
    Take $r=\delta^{d/2}$ in Lemma~\ref{le:perimeterbound} to find
    \[
    \GPer_\epsilon(B_\delta(0)) \leqslant C \left( \delta^{d-1} \int_{|h|<\delta^{d/2}} \rho_\epsilon(h)\dd{h}  + \delta^{d/2}\right) \xrightarrow{\delta\to 0} 0.\qedhere
    \]
\end{proof}

We now show that the lower semicontinuous envelope of $\GPer_\epsilon^c$ on $\Mcal_d$, i.e., the largest lower semicontinuous function below $\GPer_\epsilon^c$, is equal to $\GPer_\epsilon$.

\begin{prop}\label{prop:connectedrelaxation}
    For all $\epsilon>0$, it holds that
    \[
    \mathrm{lsc}\,\GPer_\epsilon^c = \GPer_\epsilon.
    \]
\end{prop}

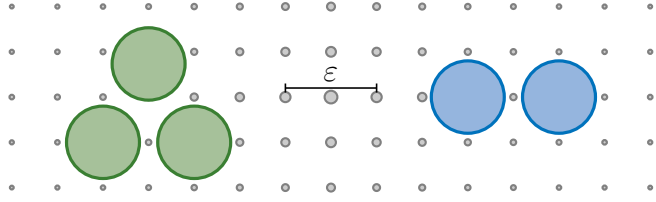
\begin{figure}[ht]
\centering
\begin{tikzpicture}[scale=1.2]               
     \foreach \k in {-1.5,-1.0,...,5.5}{
        \foreach \l in {-0.5,0.0,...,1.5}{
            \draw[color=gray, fill=gray!40, thick] (\k,\l) circle ({0.07/(1+0.5*abs(\k-2)+0.6*abs(\l-0.5))});
        }
    }
    \draw[semithick] (1.5,0.6) -- (2.5,0.6);
    \draw[semithick] (1.5,0.55) -- (1.5,0.65);
    \draw[semithick] (2.5,0.55) -- (2.5,0.65);
    \draw (2.0,0.75) node {\large $\varepsilon$};
    \draw[color=OliveGreen, fill=OliveGreen!40, very thick] (0,0.866) circle (0.4);
    \draw[color=OliveGreen, fill=OliveGreen!40, very thick] (-0.5,0) circle (0.4);
    \draw[color=OliveGreen, fill=OliveGreen!40, very thick] (0.5,0) circle (0.4);
    \draw[color=RoyalBlue, fill=RoyalBlue!40, very thick] (3.5,0.5) circle (0.4);
    \draw[color=RoyalBlue, fill=RoyalBlue!40, very thick] (4.5,0.5) circle (0.4);     
\end{tikzpicture}
\caption{The construction in the proof of Proposition \ref{prop:connectedrelaxation}. Adding the union of balls $F_n$ depicted in gray to the set $E$ with two $\varepsilon$-connected components depicted in green and blue makes $E \cup F_n$ $\varepsilon$-indecomposable.}\label{fig:connectify}
\end{figure}

\begin{proof}
    Note that $\GPer_\epsilon \leqslant\GPer_\epsilon^c$ and one can show that $\GPer_\epsilon$ is lower semicontinuous with the help of Fatou's lemma. To show that $\GPer_\epsilon$ is the largest lower semicontinuous function below $\GPer_\epsilon^c$, it suffices to find, for every $E \in \Mcal_d$, a sequence $E_n \to E$ in $\Mcal_d$ such that
    \[
    \liminf_{n \to \infty} \GPer_\epsilon^c(E_n) \leqslant\GPer_\epsilon(E).
    \]
    In the case where $E$ is $\epsilon$-indecomposable, we can simply take $E_n:=E$ for each $n \in \N$. Otherwise, we define $E_n:=E \cup F_n$ with 
    \begin{equation}\label{eq:fn}
    F_n := \bigcup_{a \in \Z^d} B_{r_a/n}\left(\frac{\epsilon}{2}a\right),
    \end{equation}
    see Figure \ref{fig:connectify}, where $r_a>0$ is chosen such that $\sum_{a \in \Z^d}r_a^d<\infty$ and
    \begin{equation}\label{eq:majorant}
    \GPer_\epsilon(B_{r_a/n}(0)) \leqslant\frac{1}{1+|a|^{d+1}} \quad \text{for all $a \in \Z^d$ and $n \in \N$.}
    \end{equation}
    The latter is possible in light of Corollary~\ref{cor:perimeterball}. Then, it follows that
    \[
    \lim_{n \to \infty}|F_n| \leqslant\lim_{n \to \infty}\frac{1}{n^d}\sum_{a \in \Z^d} r_a^d =0
    \]
    and, by the dominated convergence theorem with the majorant from \eqref{eq:majorant} and Corollary~\ref{cor:perimeterball} again,
    \[
    \lim_{n \to \infty}\GPer_{\epsilon}(F_n) \leqslant\lim_{n \to \infty}\sum_{a \in \Z^d} \GPer_\epsilon(B_{r_a/n}(0))=\sum_{a \in \Z^d} \lim_{n \to \infty}\GPer_\epsilon(B_{r_a/n}(0)) =0.
    \]
     We conclude that $E_n \to E$ as $n\to \infty$ and, since $E_n$ is clearly $\epsilon$-indecomposable by construction,
    \[
    \liminf_{n \to \infty} \GPer_\epsilon^c(E_n) = \liminf_{n \to \infty} \GPer_\epsilon(E_n) \leqslant \GPer_\epsilon(E) +\lim_{n \to \infty} \GPer_\epsilon(F_n) = \GPer_\epsilon(E).\qedhere
    \]
\end{proof}

We can now prove the corresponding $\Gamma$-convergence result, in which the $\epsilon$-indecomposability constraint is lost in the limit in any dimension. The proof is a direct combination of  Proposition~\ref{prop:connectedrelaxation} with the results in \cite{Pon04}, which cover the localization of $\GPer_\epsilon$ to the classical perimeter $P$. We refer the reader to \cite{braides2002gamma, DM92} for a general introduction to $\Gamma$-convergence.

\begin{thm}\label{thm:connectedperimeterlocalization}
    It holds that
    \[
    \Gamma\text{-}\lim_{\epsilon \to 0} \GPer_\epsilon^c = K_d\Per,
    \]
    with $K_d:=|\S^{d-1}|^{-1}\int_{\S^{d-1}}|e_1\cdot \sigma|\dd{\Hcal^{d-1}(\sigma)}$.
\end{thm}
\begin{proof}
    It is well-known that if the $\Gamma$-limit exists, then $\Gamma\text{-}\lim_{\epsilon \to 0} \GPer_\epsilon^c=\Gamma\text{-}\lim_{\epsilon \to 0} \mathrm{lsc}\,\GPer_\epsilon^c$ (cf.~\cite[Proposition 1.32]{braides2002gamma}). Therefore, thanks to Proposition~\ref{prop:connectedrelaxation}, we only need to prove that
    \[
    \Gamma\text{-}\lim_{\epsilon \to 0} \GPer_\epsilon = K_d \Per.
    \]
    For the recovery sequence, we consider a finite perimeter set $E \in \Mcal_d$ and take the constant sequence $E_\epsilon = E$ for all $\epsilon >0$. In the case that $|E|<\infty$ we have that $\mathds{1}_{E} \in \BV(\R^d)$ and hence, using \cite[Corollary~1]{Pon04} with $\Omega=\R^d$, we find:
    \[
    \lim_{\epsilon \to 0} \GPer_\epsilon(E)=\lim_{\epsilon \to 0} |\mathds{1}_{E}|_{W^{\rho_\epsilon,1}(\R^d)}=K_{d}|D\mathds{1}_E|(\R^d)=K_d\Per(E).
    \]
    If $|E|=\infty$, then $|E^c|<\infty$ and we can argue as follows
    \[
    \lim_{\epsilon \to 0} \GPer_\epsilon(E)=\lim_{\epsilon \to 0} |1-\mathds{1}_{E}|_{W^{\rho_\epsilon,1}(\R^d)}=K_{d}|D(1-\mathds{1}_E)|(\R^d)=K_d\Per(E).
    \]
    For the liminf-inequality, we consider a sequence $E_\epsilon \to E$ and find for every $R>0$, using \cite[Corollary~8]{Pon04} applied to $\Omega=B_R(0)$, that
    \begin{align*}
        \liminf_{\epsilon \to 0} \GPer_\epsilon(E_\epsilon) &=  \liminf_{\epsilon \to 0}\int_{\R^d}\int_{\R^d}\frac{|\mathds{1}_{E_\epsilon}(x)-\mathds{1}_{E_\epsilon}(y)|}{|x-y|}\rho_\epsilon(x-y)\dd{x}\dd{y} \\
        &\geqslant \liminf_{\epsilon \to 0}\int_{B_R(0)}\int_{B_R(0)}\frac{|\mathds{1}_{E_\epsilon}(x)-\mathds{1}_{E_\epsilon}(y)|}{|x-y|}\rho_\epsilon(x-y)\dd{x}\dd{y} \\
        &\geqslant K_{d}|D\mathds{1}_E|(B_R(0)).
    \end{align*}
    Letting $R \to \infty$ yields the liminf-inequality.
\end{proof}

While the previous $\Gamma$-convergence result is interesting in its own right, it does not immediately provide interesting statements about minimizers. Indeed, the minimizer of $\GPer_\epsilon^c$ and $\Per$ over all of $\Mcal_d$ is simply the empty set. In order to have nontrivial minimizers and convergence of minimizers, we can add a mass constraint in the functional and only consider subsets of a given bounded Lipschitz domain $\Omega \subset \R^d$. Precisely, for $m \in (0,|\Omega|)$, we define $\Fcal_\epsilon:\Mcal_d\to [0,\infty]$ as
\[
\Fcal_\epsilon(E):=\begin{cases}
    \GPer^c_\epsilon(E) &\text{if $E \subset \Omega$ and $|E|=m$},\\
    \infty &\text{else},
\end{cases}
\]
and the local functional $\Fcal:\Mcal_d \to [0,\infty]$ as
\[
\Fcal(E):=\begin{cases}
    K_d\Per(E) &\text{if $E \subset \Omega$ and $|E|=m$},\\
    \infty &\text{else}.
\end{cases}
\]
We can prove the following result with a slight adaptation to Theorem~\ref{thm:connectedperimeterlocalization}.
\begin{cor}
    It holds that
    \[
    \Gamma\text{-}\lim_{\epsilon \to 0} \Fcal_\epsilon = \Fcal
    \]
    and $(\Fcal_\epsilon)_\epsilon$ is equi-coercive with respect to the convergence in $\Mcal_d$. In particular, any sequence $(E_\epsilon)_\epsilon \subset \Mcal_d$ of almost minimizers of $(\Fcal_\epsilon)_\epsilon$ converges up to subsequence to a minimizer of $\Fcal$, that is, to a solution of the minimization problem
    \[
    \inf\left\{ \Per(E) \,\middle|\, E \subset \Omega \ \text{measurable}, |E|=m\right\}.
    \]
\end{cor}
\begin{proof}
    \textit{Equi-coercivity:} Let $(E_\epsilon)_\epsilon \subset \Mcal_d$ be a sequence such that $M:=\sup_\epsilon \Fcal_\epsilon (E_\epsilon)<\infty$. Then, we find that $E_\epsilon \subset \Omega$ for all $\epsilon >0$ and
    \[
    \int_{\Omega}\int_{\Omega}\frac{|\mathds{1}_{E_\epsilon}(x)-\mathds{1}_{E_\epsilon}(y)|}{|x-y|}\rho_\epsilon(x-y)\dd{x}\dd{y} \leqslant |\mathds{1}_{E_\epsilon}|_{W^{\rho_\epsilon,1}(\R^d)} = \GPer_\epsilon(E_\epsilon) \leqslant \Fcal_\epsilon(E_\epsilon)\leqslant M.
    \]
    By \cite[Theorem~1.2 and 1.3]{Pon03} (with the additional assumption on $\rho_\epsilon$ from \cite[Eq.~(8)]{Pon03} if $d=1$), we find that $(\mathds{1}_{E_\epsilon})_\epsilon$ converges up to subsequence to some $u$ in $L^1(\Omega)$. Clearly, it holds that $u=\mathds{1}_{E}$ for some $E \subset \Omega$ and $|E|=m$. Therefore, we find that $E_\epsilon \to E$ in $\Mcal_d$. \medskip

    \textit{$\Gamma$-convergence:} We can use the same argument as in Theorem~\ref{thm:connectedperimeterlocalization}, the only adaptation is to make sure that the recovery sequence in Proposition~\ref{prop:connectedrelaxation} can be chosen to adhere to the constraints $E_n \subset \Omega$ and $|E_n|=m$ for every $n \in \N$. This is possible by taking a finite set $A \subset \Omega$ such that $\Omega \Subset A+B_{\epsilon/2}(0)$ and defining
    \[
    F_n = \bigcup_{a \in A} B_{1/n}(a).
    \]
    Then, we set
    \[
    E_n:=(E \cup (F_n \cap \Omega)) \setminus B_{r_n}(x_0),
    \]
    where $x_0 \in E^{(1)} \setminus A$ and $r_n$ is chosen via the intermediate value theorem to get $|E_n|=m$. The latter relies on the continuity of $r \mapsto |(E \cup (F_n \cap \Omega)) \cap B_{r}(x_0)|$. Moreover, if $|F_n|$ is small enough, then $r_n < \dist(x_0, A)$ given that 
    \[
    \lim_{r \to 0} \frac{ |(E \cup (F_n \cap \Omega)) \cap B_{r}(x_0)|}{|B_r(x_0)|} =1.
    \]
    In particular, $r_n \to 0$ as $n \to \infty$ and $E_n$ is $\epsilon$-indecomposable. The indecomposability relies on the observation that $A$ can not be partitioned into two sets $A_1,A_2$ with $\dist(A_1,A_2)>\epsilon$, since then $A_1+B_{\epsilon/2}(0)$ and $A_2+B_{\epsilon/2}(0)$ would give a nontrivial disjoint open cover of the connected set $\Omega$. Additionally, the convergence $E_n \to E$ is clear, while we also have
    \[
    \GPer^c_\epsilon(E_n) \leqslant \GPer_\epsilon(E) + \GPer_\epsilon(B_{r_n}(x_0)) + \sum_{a \in A} \GPer_\epsilon\left(B_{1/n}\left(a\right) \cap \Omega\right).
    \]
    Now, using that $\GPer_\epsilon\left(B_{1/n}\left(a\right) \cap \Omega\right) \to 0$ by arguing as in Corollary~\ref{cor:perimeterball}, we can obtain:
    \[
    \limsup_{n \to \infty} \GPer_\epsilon^c(E_n) \leqslant \GPer_\epsilon(E).\qedhere
    \]
\end{proof}

When we let the horizon $\epsilon$ instead diverge to $\infty$, we can recover the fractional perimeter as well. For this, we assume again that $\rho_\epsilon$ is radial with $\esssupp \rho_\epsilon = \overline{B_\varepsilon(0)}$ for all $\epsilon >0$, without requiring the normalization \eqref{eq:mean}, and additionally that as $\epsilon \to \infty$
\begin{equation}\label{eq:epsiloninftyhypothesis}
\min\{1,|\cdot|^{-1}\}\,\rho_\epsilon(\cdot) \to \min\{1,|\cdot|^{-1}\}\,\frac{1}{|\cdot|^{d+\alpha-1}} \quad \text{pointwise a.e.~and in $L^1(\R^d)$.}
\end{equation}
\begin{example}
    One example that would satisfy \eqref{eq:epsiloninftyhypothesis} is 
    \[
    \rho_\epsilon(\cdot)=\overline{\rho}_\epsilon(|\cdot|), \quad \overline{\rho}_\epsilon(r):=\mathds{1}_{[0,\epsilon]}(r)\frac{1}{r^{d+\alpha-1}}
    \]
    for $\alpha \in (0,1)$. Another kernel that does not have purely fractional singularity at the origin would be given by
    \[
    \overline{\rho}_\epsilon(r):=\mathds{1}_{[0,\epsilon/2]}(r)\frac{\log(\epsilon)}{\log(\epsilon/r)}\frac{1}{r^{d+\alpha-1}}=\mathds{1}_{[0,\epsilon/2]}(r)\frac{1}{1-\log(r)/\log(\epsilon)}\frac{1}{r^{d+\alpha-1}},
    \]
    see~\cite[Example~4.4\,(b) and Lemma~4.5]{cueto2024gamma}.
\end{example}
We have the following convergence result.
\begin{thm}
    It holds that
    \[
    \Gamma\text{-}\lim_{\epsilon \to \infty} \GPer^c_\epsilon = \GPer_\alpha.
    \]
\end{thm}
\begin{proof}
    The liminf-inequality is a straightforward consequence of Fatou's lemma by exploiting the pointwise a.e.~convergence from \eqref{eq:epsiloninftyhypothesis}. For the recovery sequence, we first prove that
    \begin{equation}\label{eq:kernerlperimeterbound}
    |\GPer_\epsilon(E)-\GPer_\alpha(E)| \leqslant C\|\mathds{1}_E\|_{\BV(\R^d)}\int_{\R^d}\min\{1,|h|^{-1}\}|\rho_\epsilon({h})-|h|^{-(d+\alpha-1)}|\dd{h}.
    \end{equation}
    For $\phi \in W^{1,1}(\R^d)$, we can compute that
    \begin{align*}
    ||\phi|_{W^{\rho_\epsilon,1}(\R^d)}-|\phi|_{W^{\alpha,1}(\R^d)}| &\leqslant \int_{\R^d}\int_{\R^d} |\phi(x+h)-\phi(x)|\dd{x}\,\abs{h}^{-1}|\rho_\epsilon({h})-|h|^{-(d+\alpha-1)}|\dd{h} \\
    &\leqslant C\|\phi\|_{W^{1,1}(\R^d)}\int_{\R^d}\min\{1,|h|^{-1}\}|\rho_\epsilon({h})-|h|^{-(d+\alpha-1)}|\dd{h},
    \end{align*}
    by using that 
    \[
    \int_{\R^d} |\phi(x+h)-\phi(x)|\dd{x} \leqslant C\min\{1,|h|\}\|\phi\|_{W^{1,1}(\R^d)}.
    \]
    Hence, by approximating $\mathds{1}_{E}$ with $W^{1,1}$-functions strictly, we deduce \eqref{eq:kernerlperimeterbound}.

    Now, we take $E \in \Mcal_d$ such that $\GPer_\alpha(E)<\infty$. By \cite[Corollary~7.9]{Leo23}, we deduce that either $E$ or $E^c$ has finite measure. Therefore, by \cite[Theorem~1.1]{Lom18}, we can take a sequence $E_n$ of smooth sets such that either $E_n$ or $E_n^c$ is bounded for all $n \in \N$, $E_n \to E$ and $\GPer_\alpha(E_n) \to \GPer_\alpha(E)$. In the case that $E_n$ is bounded, we find, using \eqref{eq:kernerlperimeterbound} and \eqref{eq:epsiloninftyhypothesis}, that
    \[
    \lim_{n \to \infty}\lim_{\epsilon \to \infty}\GPer^c_\epsilon(E_n)=\lim_{n \to \infty}\lim_{\epsilon \to \infty}\GPer_\epsilon(E_n) = \lim_{n \to \infty}\GPer_\alpha(E_n) = \GPer_\alpha(E),
    \]
    where the first equality uses that $E_n$ is bounded and hence $\epsilon$-indecomposable for large enough $\epsilon$. If $E_n^c$ is bounded, we may deduce the same by using that $\GPer_\alpha(E_n)=\GPer_\alpha(E_n^c)$ and $\GPer_\epsilon(E_n)=\GPer_\epsilon(E_n^c)$. By choosing a suitable diagonal sequence $E_{n_\epsilon}$ we obtain the desired recovery sequence.
\end{proof}

\section{Extreme points and decomposability in \texorpdfstring{$\BV^\rho(\R^d)$}{BVrho} and \texorpdfstring{$\BV^\alpha(\R^d)$}{BValpha}.}\label{sec:extdecdistributional}

We now turn to the nonlocal variation and perimeter defined in a distributional sense and study the extreme points and decomposability properties in this context. We start off with some preliminaries on the functions of bounded nonlocal variation, and subsequently investigate decomposability and extremality. It turns out that nonlocal decomposability in this setting is equivalent to $2\epsilon$-decomposability studied in Section~\ref{sec:nonlocalper}, at least for regular enough sets (cf.~Theorem~\ref{thm:epsdecomp}). In contrast, the extreme points are not indicator functions like for the Gagliardo perimeter and also not related to the nonlocal notion of decomposability. Instead they are obtained via the extreme points in $\BV$ using a suitable isomorphism, see~Theorem~\ref{thm:extepsdistributional} and \ref{thm:extalphadistributional}.

\subsection{Functions of bounded nonlocal variation}\label{sec:distrBV}

Here, we introduce the functions of bounded nonlocal variation and some of their properties, which are needed in the paper. These spaces have only been considered in the fractional case \cite{ComSte19,ComSte23, brue2022distributional}, but we extend it to more general kernels, utilizing the recent tools developed in \cite{BelMorSch24}.

Throughout Section~\ref{sec:extdecdistributional} and \ref{sec:gammadistributional} we consider a general radial kernel $\rho:\R^d \setminus\{0\} \to [0,\infty)$ that satisfies for some $\eta>0$
\begin{equation}\label{eq:assumption}
    \inf_{\overline{B_\eta(0)}}\rho>0 \quad \text{and} \quad \rho \cdot \min\{1,\abs{\cdot}^{-1}\} \in L^1(\R^d).
\end{equation}
One can define the associated nonlocal gradient $D_\rho u$ of $u \in C_c^{\infty}(\R^d)$ as
\[
D_\rho u(x):= \int_{\R^d} \frac{u(y)-u(x)}{\abs{y-x}}\frac{y-x}{\abs{y-x}} \rho(y-x)\dd y,
\]
and the nonlocal divergence of $p \in C_c^{\infty}(\R^d;\R^d)$ as
\[
\div_\rho p(x):= \int_{\R^d} \frac{p(y)-p(x)}{\abs{y-x}}\cdot\frac{y-x}{\abs{y-x}}\rho(y-x) \dd y.
\]
These operators are well-defined and define smooth, bounded and integrable functions (cf.~\cite{BelMorSch24}). Moreover, they are dual in the sense that
\begin{equation}\label{eq:intbyparts}
    \int_{\R^d} u\div_\rho p\dd x = -\int_{\R^d} D_\rho u \cdot p \dd x.
\end{equation}Similarly to the fractional case of \cite{ComSte19}, we introduce the following space.
\begin{defi}[Functions of bounded nonlocal variation]
The space $\BV^\rho(\R^d)$ is defined as the functions $u \in L^1(\R^d)$ such that
\[
\TV_\rho (u):=\sup\left\{ \int_{\R^d} u \div_\rho p \dd x \,\middle\vert\, p \in C_c^{\infty}(\R^d;\R^d),\, \|p\|_{L^{\infty}(\R^d;\R^d)}\leq1\right\}<\infty,
\]
endowed with the norm
\[
\|u\|_{\BV^\rho(\R^d)}:=\|u\|_{L^1(\R^d)}+\TV_\rho(u).
\]
\end{defi}
\begin{example}[Fractional case]\label{ex:frac}
    The kernel
    \[
    \rho(\cdot)=c_{d,\alpha}\frac{1}{|\cdot|^{d-1+\alpha}} \quad \text{with} \quad c_{d,\alpha}:=2^\alpha\pi^{-d/2}\frac{\Gamma((d+\alpha+1)/2)}{\Gamma((1-\alpha)/2)}
    \]
    with $\alpha \in (0,1)$ leads to the fractional gradient $D^\alpha$ and $\BV^\alpha$, see~e.g.~\cite{ComSte19,ComSte23, brue2022distributional}.
\end{example}
One can deduce that for $u \in \BV^\rho(\R^d)$, there exists a measure $D_\rho u \in \M(\R^d;\R^d)$ such that
\[
\int_{\R^d} u \div_\rho p \dd x = - \int_{\R^d} p \cdot \dd D_\rho u \ \text{ for all }p \in C^\infty_c(\R^d;\R^d), \ \text{ and }\TV_\rho(u) = |D_\rho u|(\R^d).
\]
In view of \eqref{eq:intbyparts}, the measure $D_\rho u$, which we call the \textit{nonlocal variation measure} of $u$, is a natural extension of the nonlocal gradient to the space $\BV^\rho(\R^d)$.

An important tool we need is the potential $Q_\rho:\R^d \setminus \{0\} \to [0,\infty)$ defined by
\[
Q_\rho(z) = \int_{|z|}^\infty \frac{\overline{\rho}(t)}{t}\dd t \quad \text{for $z \in \R^d \setminus\{0\}$} \quad \text{with} \ \rho(\cdot)=\overline{\rho}(|\cdot|),
\]
which is a locally integrable function that satisfies
\begin{equation}\label{eq:potential}
D_\rho u = Q_\rho * Du = D(Q_\rho * u) \quad \text{for all $u \in C_c^{\infty}(\R^d)$},
\end{equation}
see~\cite[Proposition~2.6]{BelMorSch24}. In the case that $\rho \in L^1(\R^d)$, it holds that $Q_\rho \in L^1(\R^d)$, while in the fractional case, it holds that $Q_\rho$ coincides with the Riesz potential $I_{1-\alpha} \propto 1/|\cdot|^{d-1+\alpha}$, which is not integrable. In the former case, we can extend \eqref{eq:potential} to the $\BV$-setting as follows.
\begin{lemma}\label{le:fromnonlocaltolocal}
If $\rho \in L^1(\R^d)$, then the linear map $\Qcal_\rho:\BV^\rho(\R^d) \to \BV(\R^d), \ u \mapsto Q_\rho *u$ is bounded and satisfies
\[
D_\rho u = D (\Qcal_\rho u) \quad \text{for all $u \in \BV^\rho(\R^d)$.}
\]
Moreover, if $u \in \BV(\R^d)$ then $D_\rho u = Q_\rho * Du \in L^1(\R^d)$.
\end{lemma}
\begin{proof}
    For $u \in \BV^\rho(\R^d)$ it follows that $\Qcal_\rho u \in L^1(\R^d)$ by Young's convolution inequality. Moreover, one can compute with Fubini's theorem that
    \begin{align*}
        \int_{\R^d}\Qcal_\rho u \div p \dd x &= \int_{\R^d} u \Qcal_\rho(\div p) \dd x\\
        &= \int_{\R^d} u \div_\rho p \dd x =-\int_{\R^d} p\dd D_\rho u.
    \end{align*}
    Hence, $D(\Qcal_\rho u)=D_\rho u$, after which the boundedness follows. The second statement can be proven with a similar duality argument.
\end{proof}
There is also an inverse of $\Qcal_\rho$, which requires some additional assumptions on $\rho$. Specifically, we assume that $\rho$ has compact support and satisfies the conditions \ref{itm:h1}-\ref{itm:upper} from \cite{BelMorSch24}. In particular, given $\eta>0$ as in \eqref{eq:assumption}, $\nu >0$ and $0<\sigma \leqslant \gamma <1$, the following holds:
\begin{enumerate}[label = (H\arabic*)]
\item \label{itm:h1} The function $f_\rho:(0,\infty)\to \R, \ r \mapsto r^{d-2}\overline{\rho}(r)$ is nonincreasing on $(0,\infty)$ and $r \mapsto r^\nu f_\rho(r)$ is nonincreasing on $(0,\eta)$;

\item \label{itm:derivatives} $f_\rho$ is smooth on $(0,\infty)$ and for every $k\in \N$ there exists a $C_k>0$ with
\[
\left|\frac{d^k}{dr^k}f_\rho(r)\right| \leqslant C_k \frac{f_\rho(r)}{r^k} \quad \text{for $r\in (0, \eta)$};
\]

\item \label{itm:lower} the function $r \mapsto r^{d+\sigma-1}\overline{\rho}(r)$ is almost nonincreasing on $(0,\eta)$; 

\item \label{itm:upper} the function $r \mapsto r^{d+\gamma-1}\overline{\rho}(r)$ is almost nondecreasing on $(0,\eta)$.
\end{enumerate}
Then, one can define the operator
\[
\Pcal_\rho: \Scal(\R^d) \to \Scal(\R^d), \quad \Pcal_\rho v := \left(\widehat{v}/\widehat{Q}_\rho\right)^{\vee},
\]
see~\cite[Lemma~2.12]{cueto2024gamma}, which is the inverse of $\Qcal_\rho$ on $\Scal(\R^d)$. To show that $\Pcal_\rho$ extends to the $\BV$-setting requires additional work, given that Fourier techniques do not directly work in this setting. A key tool is the nonlocal fundamental theorem of calculus, which we state here for the reader's convenience.
\begin{thm}[{\cite[Theorem~5.2]{BelMorSch24}}]\label{thm:nftocsmooth}
     Let $\rho$ have compact support and satisfy \ref{itm:h1}-\ref{itm:upper}. Then, there exists a vector radial function $V_\rho \in C^{\infty}(\R^d\setminus \{0\},\R^d) \cap L^1_{\rm loc} (\R^d,\R^d)$ such that for all $u \in C_c^{\infty}(\R^d)$,
\begin{equation}\label{eq:ftoc}
u(x) = \int_{\R^d} V_\rho(x-y) \cdot D_\rho u(y)\dd y \quad \text{for all $x \in \R^d$}.
\end{equation}
Moreover, there is a constant $C=C(d,\rho)>0$ such that for all $z \in B_{\eta}(0) \setminus\{0\}$,
\begin{equation}\label{eq:Vrhobound}
\abs{V_\rho(z)}\leqslant \frac{C}{\abs{z}^{2d-1}\rho(z)} \quad \text{and} \quad \abs{\nabla V_\rho(z)} \leqslant \frac{C}{\abs{z}^{2d}\rho(z)}.
\end{equation}
\end{thm}
We can now prove the counterpart to Lemma~\ref{le:fromnonlocaltolocal}. Its proof, and also Lemma~\ref{le:Vrhoestimates} below, utilize some techniques from Fourier theory, for which we refer reader to \cite{Gra14}. We use
\[
\widehat{u}(\xi) := \int_{\R^d} e^{-2\pi i x \cdot \xi}u(x)\dd x \quad \text{for $\xi \in \R^d$ and $u \in L^1(\R^d)$}
\]
to denote the Fourier transform of $u$ and write $u^{\vee}$ for its inverse Fourier transform. The same notation is also used for the extension of the Fourier transform to the space of tempered distributions.
\begin{lemma}\label{le:fromlocaltononlocal}
    Let $\rho$ have compact support and satisfy \ref{itm:h1}-\ref{itm:upper}. Then, $\Pcal_\rho$ extends to a bounded linear operator from $\BV(\R^d)$ to $\BV^\rho(\R^d)$ such that $\Pcal_\rho =(\Qcal_\rho)^{-1}$. In particular, it holds that
    \[
    Dv = D_\rho (\Pcal_\rho v) \quad \text{for all $v \in \BV(\R^d)$.}
    \]
\end{lemma}
\begin{proof}
    Consider the function $V_\rho$, then from its Fourier transform in \cite[Proposition~4.4]{BelMorSch24} one can deduce that $\widehat{\div V_\rho} = 1/\widehat{Q}_\rho$ when it is understood as a tempered distribution. Via integration by parts, it then follows that
    \begin{equation}\label{eq:pcalalternative}
    \Pcal_{\rho} u = V_\rho * Du \quad \text{for all $u \in C_c^{\infty}(\R^d)$.}
    \end{equation}
    In light of \cite[Lemma~2.6]{cueto2024gamma}, we find that $\partial^{\gamma}(1/\widehat{Q}_\rho)$ is integrable for all multi-indices $\gamma \in \N^d_0$ with $|\gamma| \geqslant d+1$. In turn, by using that the inverse Fourier transform is bounded from $L^1$ to $L^\infty$, this implies that
    \[
    (2\pi i \cdot)^\gamma \div V_\rho
    \]
    is a bounded function. The vector-radiality of $V_\rho$ then shows that $\div V_\rho$ decays faster than $|\cdot|^{-N}$ at infinity for any $N \in \N$. In particular, for a radial cut-off function $\chi \in C_c^{\infty}(\R^d)$ with $\chi \equiv 1$ on $B_1(0)$, we deduce that both $\chi V_\rho$ and $\div((1-\chi) V_\rho)$ are integrable functions. From \eqref{eq:pcalalternative} we then find the representation
    \[
    \Pcal_{\rho} u = (\chi V_\rho) * Du + \div((1-\chi) V_\rho) * u  \quad \text{for all $u \in C_c^{\infty}(\R^d)$,}
    \]
    which extends continuously to $\BV(\R^d)$ by Young's convolution inequality. Showing that this extension is indeed the inverse of $\Qcal_\rho$ can be reduced to the smooth case via duality arguments.
\end{proof}
We can also extend the fundamental theorem of calculus to $\BV^\rho$, in order to obtain an alternative characterization of the extreme points later.
\begin{prop}\label{prop:nftoc}
   Let $\rho$ have compact support and satisfy \ref{itm:h1}-\ref{itm:upper}. Then, for every $u \in \BV^\rho(\R^d)$ such that $D_\rho u$ is a compactly supported measure, we have that
    \[
    u(x) = \int_{\R^d} V_\rho(x-y)\cdot \dd D_\rho u(y) \quad \text{for a.e.~$x \in \R^d$.}
    \]
\end{prop}
\begin{proof}
 We first show that for any $\mu \in \M(\R^d;\R^d)$ with compact support, it holds that
    \[
    (V_\rho * \mu) (x):= \int_{\R^d} V_\rho(x-y)\cdot \dd \mu(y)=\int_{\supp \mu} V_\rho(x-y)\cdot \dd \mu(y),
    \]
    is a locally integrable function that is bounded away from the support of $\mu$. Indeed, for any bounded open set $O \subset \R^d$ with $\supp \mu \Subset O$, we find that
    \begin{align*}
    \int_O \int_{\supp \mu} |V_\rho(x-y)|\dd |\mu|(y)\dd x &= \int_{\supp \mu} \int_{O} |V_\rho(x-y)|\dd x\dd |\mu|(y)\\
    &\leqslant\int_{\supp \mu} \int_{O-\supp \mu} |V_\rho(z)|\dd z\dd |\mu|(y)\\
    &\leqslant \int_{\supp \mu} \|V_\rho\|_{L^1(O-\supp \mu;\R^d)}\dd |\mu|(y) <\infty,
    \end{align*}
    and for a.e.~$x \in (\supp \mu)^c$ we find with $\epsilon=\dist(x,\supp \mu)$
    \begin{align*}
        \int_{\supp \mu} |V_\rho(x-y)|\dd |\mu|(y) \leqslant \int_{\supp \mu} \|V_\rho\|_{L^\infty(B_\epsilon(0)^c;\R^d)} d |\mu|(y),
    \end{align*}
    which is uniformly bounded for all $\epsilon\geq\epsilon_0>0$ due to Lemma~\ref{le:Vrhoestimates} below. Given these bounds, the following computation for $\phi \in C_c^{\infty}(\R^d)$ using Fubini's theorem is justified
    \begin{align*}
        \int_{\R^d}\int_{\R^d} V_\rho(x-y)\cdot \dd D_\rho u(y) \phi(x) \dd x & = \int_{\R^d} \int_{\R^d} -V_\rho(y-x)\phi(x)\dd x \cdot \dd D_\rho u(y) \\
        &= \int_{\R^d} -(V_\rho * \phi)(y) \cdot \dd D_\rho u(y)\\
        & = \int_{\R^d} \Div_\rho (V_\rho * \phi) (y) u(y)\dd y,
    \end{align*}
    where we have used nonlocal integration by parts for the bounded Lipschitz function $(V_\rho * \phi)$, which is holds due to an argument similar to the fractional case in~\cite[Proposition~2.7]{ComSte23}. If $\Div_\rho (V_\rho * \phi) = \phi $, then we are done. To show this, we take $\psi \in C_c^{\infty}(\R^d)$ and compute again with integration by parts and Theorem~\ref{thm:nftocsmooth} that
    \begin{align*}
        \int_{\R^d} \Div_\rho (V_\rho * \phi)(x) \psi(x) \dd x & = -\int_{\R^d} (V_\rho * \phi)(x) D_\rho \psi (x) \dd x \\
        &= \int_{\R^d} \phi(x) (V_\rho * D_\rho \psi)(x)\dd x \\
        & = \int_{\R^d} \phi(x)\psi(x)\dd x.\qedhere
    \end{align*}
\end{proof}

Another interesting consequence of the fundamental theorem of calculus is the following Poincar\'e inequality and compactness result.
\begin{lemma}\label{le:poinccompact}
    Let $\rho$ have compact support and satisfy \ref{itm:h1}-\ref{itm:upper} and let $R>0$. Then, there exists a constant $C=C(d,\rho,R)>0$ such that
    \[
    \|u\|_{L^1(\R^d)} \leqslant C|D_\rho u|(\R^d) \quad \text{for all $u \in \BV_\rho(\R^d)$ with $\supp u \subset B_R(0)$.}
    \]
    Moreover, if $(u_n)_n \subset \BV_\rho(\R^d)$ is a sequence such that $\supp u_n \subset B_R(0)$ for all $n \in \N$ and 
    \[
    \sup_n |D_\rho u_n|(\R^d)<\infty,
    \]
    then $u_n \to u$ in $L^1(\R^d)$ up to a non-relabeled subsequence for some $u \in \BV^\rho(\R^d)$.
\end{lemma}
\begin{proof}
    Using Proposition~\ref{prop:nftoc} and the fact that $\supp D_\rho u \subset B_R(0)+\supp \rho$, we deduce that
    \begin{align*}
     \|u\|_{L^1(\R^d)}=\|u\|_{L^1(B_R(0))}&= \|V_\rho * D_\rho u\|_{L^1(B_R(0))} \\
     &\leqslant \|V_\rho\|_{L^1(B_{2R}(0)+\supp \rho;\R^d)}|D_\rho u|(\R^d)=C|D_\rho u|(\R^d).
    \end{align*}
    Using a similar estimate for $|\zeta| \leqslant R$, we find that
    \begin{align*}
    &\lim_{\zeta \to 0} \sup_n \|u_n(\cdot)-u_n(\cdot + \zeta)\|_{L^1(\R^d)} \\
    &\qquad \leqslant \lim_{\zeta \to 0} \|V_\rho(\cdot)-V_\rho(\cdot+\zeta)\|_{L^1(B_{3R}(0)+\supp \rho;\R^d)} \sup_n |D_\rho u_n|(\R^d) =0.
    \end{align*}
    Hence, the convergence up to subsequence to an $u \in L^1(\R^d)$ follows from the Fr\'echet-Kolmogorov criterion. We deduce via duality that $D_\rho u_n$ converges weak* to $D_\rho u$ in $\Mcal(\R^d;\R^d)$ so that
    \[
    |D_\rho u|(\R^d) \leqslant \liminf_{n \to \infty} |D_\rho u_n|(\R^d) <\infty,
    \]
    that is, $u \in \BV^\rho(\R^d)$.
\end{proof}
To prove the previous result also in the setting of vanishing horizons in Section~\ref{sec:gammadistributional}, we prepare the following bounds on the tail of $V_\rho$.

\begin{lemma}\label{le:Vrhoestimates}
     Let $\rho$ have compact support and satisfy \ref{itm:h1}-\ref{itm:upper}. Then, there exists a Schwartz function $\psi \in \Scal(\R^d;\R^d)$ and a constant $c>0$ such that
    \[
    V_\rho(z) = c\frac{z}{\abs{z}^d}+\psi(z) \quad \text{for all $z \in B_1(0)^c$.}
    \]
    In particular, there is a $C>0$ such that
    \[
    |V_\rho(z)|+|z||\nabla V_\rho(z)| \leqslant C \max\left\{\frac{1}{|z|^{d-\sigma}},\frac{1}{|z|^{d-1}}\right\} \quad\text{for all $z \not =0$}.
    \]
\end{lemma}
\begin{proof}
    Following the proof of \cite[Theorem~5.2]{BelMorSch24}, we define for some smooth radial cut-off function $\chi \in C_c^{\infty}(\R^d)$ with $\chi=1$ around the origin
    \[
    W^2_\rho(\xi):=(1-\chi(\xi))\widehat{V}_\rho(\xi)=(1-\chi(\xi))\frac{-i\xi}{2\pi|\xi|^2\widehat{Q}_\rho(\xi)}.
    \]
    The proof of \cite[Theorem~5.2]{BelMorSch24} shows that the inverse Fourier transform of this function decays faster than any polynomial at infinity. A similar argument works for its derivatives, so $(W^2_\rho)^{\vee}$ agrees with a Schwartz function away from the origin. For the remaining part
    \[
    W^1_\rho(\xi):=\chi(\xi)\widehat{V}_\rho(\xi)=\chi(\xi)\frac{-i\xi}{2\pi|\xi|^2\widehat{Q}_\rho(\xi)},
    \]
    we can split it as
    \[
    W^1_\rho(\xi)=\frac{-i\xi}{2\pi|\xi|^2\widehat{Q}_\rho(0)} + \left(\frac{\chi(\xi)}{\widehat{Q}_\rho(\xi)}-\frac{1}{\widehat{Q}_\rho(0)}\right)\frac{-i\xi}{2\pi|\xi|^2}=:Y(\xi) + U_\rho(\xi).
    \]
    Note that since $\widehat{Q}_\rho(\xi)$ is analytic, nonnegative and radial, the term $U_\rho$ is smooth around the origin. Indeed, the power series expansion of $\chi(\xi)/\widehat{Q}_\rho(\xi)-1/\widehat{Q}_\rho(0)$ around $0$ only has second order terms or higher which cancel the singularity $1/|\xi|^2$. Furthermore, $U_\rho(\xi)$ agrees with a homogeneous function outside the support of $\chi$, so that an argument as in \cite[Example~2.4.9]{Gra14} shows that $U_\rho^{\vee}$ is a Schwartz function away from the origin. Finally, the inverse Fourier transform of $Y$ is exactly
    \[
    Y^{\vee}(z)=c\frac{z}{\abs{z}^d},
    \]
    for some suitable $c>0$, cf.~\cite[Lemma~B.1\,c)]{bellido2023non}. Since $V_\rho =Y^{\vee}+ (W^2_\rho)^{\vee}+U_\rho^{\vee}$, the first part follows. The second part is now immediate, since we locally have the estimate
    \[
    |V_\rho(z)|+|z||\nabla V_\rho(z)| \leqslant C\frac{1}{|z|^{2d-1}\rho(z)} \leqslant C \frac{1}{|z|^{d-\sigma}} \quad \text{for all $z \in B_\varepsilon(0)$,}
    \]
    by Theorem~\ref{thm:nftocsmooth} and \ref{itm:lower}.
\end{proof}
In the following remark, we summarize the analogs of the preceding results in the fractional case.
\newpage
\begin{rem}\label{rem:fractionalcase}
\leavevmode
\begin{itemize}
    \item[(i)] In the fractional case, the operator $\Qcal_\rho$ agrees with convolution with the Riesz potential $I_{1-\alpha}$, while $\Pcal_\rho$ agrees with the fractional Laplacian $(-\Delta)^{\frac{1-\alpha}{2}}.$ However, even though the fractional Laplacian maps $\BV(\R^d)$ into $\BV^\alpha(\R^d)$ continuously, for $u \in \BV^\alpha(\R^d)$ it only holds that $I_{1-\alpha}*u \in \BV_{\mathrm{loc}}(\R^d)$, see~\cite[Lemma~3.28]{ComSte19}. Therefore, there is no isomorphism like in Lemma~\ref{le:fromnonlocaltolocal} and \ref{le:fromlocaltononlocal}. Still, if $u \in \BV(\R^d)$, one has, by \cite[Theorem~3.18]{ComSte19}, that $D^\alpha u \in L^1(\R^d)$ with
    \begin{equation}\label{eq:translationproperty}
    D^\alpha u = I_{1-\alpha}*Du.
    \end{equation}
    \item[(ii)] The kernel in the fractional fundamental theorem of calculus can be explicitly given by $V^{\alpha}(z)=c_{d,-\alpha}z/|z|^{d+1-\alpha}$, with $c_{d,-\alpha}$ as in Example~\ref{ex:frac}, see~e.g.,~\cite[Theorem~3.12]{ComSte19}. Hence, it holds for $u \in C_c^{\infty}(\R^d)$ that
    \[
    u(x)=c_{d,-\alpha}\int_{\R^d} \frac{x-y}{|x-y|^{d+1-\alpha}} \cdot D^\alpha u(y)\dd y \quad \text{for all $x \in \R^d$}.
    \]
    If instead $u \in \BV^\alpha(\R^d)$ with $D^\alpha u$ compactly supported, then one can argue as in Proposition~\ref{prop:nftoc} to find that
    \[
    u(x)=c_{d,-\alpha}\int_{\R^d} \frac{x-y}{|x-y|^{d+1-\alpha}} \cdot \dd D^\alpha u(y) \quad \text{for a.e.~$x \in \R^d$}.
    \]

    \item[(iii)] Similar estimates and compactness results to Lemma~\ref{le:poinccompact} in the fractional case can be found in \cite[Theorem~3.9 and 3.16]{ComSte19}. \qedhere
    \end{itemize}
\end{rem}

\subsection{Nonlocal Caccioppoli perimeter and decomposability}\label{sec:distrDec}

In this section, we introduce the notion of nonlocal Caccioppoli perimeter following the fractional case in \cite{ComSte19} and investigate the decomposition properties of sets with respect to this perimeter. The main decomposition result in Theorem~\ref{thm:epsdecomp} is very similar to the Gagliardo setting, although the proof relies on completely different techniques. 

We consider a radial kernel $\rho$ satisfying \eqref{eq:assumption} and introduce the following notion analogously to the fractional case in \cite[Definition~4.1]{ComSte19}.
\begin{defi}[Nonlocal Caccioppoli perimeter] 
We define the Caccioppoli $\rho$-perimeter of a measurable set $E \subset \R^d$ as
\[
\Per_\rho(E):= \TV_\rho(\mathds{1}_E)=\sup\left\{ \int_{E} \div_\rho p \dd x \,\middle\vert\, p \in C_c^{\infty}(\R^d;\R^d),\, \|p\|_{L^{\infty}(\R^d;\R^d)}\leq1\right\}.
\]
If $\Per_\rho(E)<\infty$, or equivalently, $D_\rho \mathds{1}_{E} \in \M(\R^d;\R^d)$, then we say that $E$ is a set with finite Caccioppoli $\rho$-perimeter.
\end{defi}

A complete analysis of the Caccioppoli $\rho$-perimeter is not the main focus of this work, but we mention that many properties can be proven by following the strategy from the fractional case in \cite{ComSte19, ComSte23, brue2022distributional} and using the tools from the previous section. For example, the functional $\Per_\rho$ is lower semicontinuous with respect to convergence of sets, and we have the following immediate consequence of Lemma~\ref{le:poinccompact}.
\begin{lemma}\label{le:poinccompper}
     Let $\rho$ have compact support and satisfy \ref{itm:h1}-\ref{itm:upper} and let $R>0$. Then, there exists a constant $C=C(d,\rho,R)>0$ such that
    \[
    |E| \leqslant C\Per_\rho(E) \quad \text{for all measurable $E \subset B_R(0)$.}
    \]
    Moreover, if $(E_n)_n$ is a sequence such that $E_n \subset B_R(0)$ for all $n \in \N$ and 
    \[
    \sup_n \Per_\rho(E_n)<\infty,
    \]
    then $E_n \to E$ up to a non-relabeled subsequence for some measurable set $E \subset B_R(0)$.
\end{lemma}

We now turn to the study of decomposability of sets with respect to the Caccioppoli $\rho$-perimeter.
\begin{defi}
    We say that a set $E$ of finite Caccioppoli $\rho$-perimeter is $\Per_\rho$-decomposable, if there exists a partition $\{E_1,E_2\}$  of $E$ with $|E_1|,|E_2|>0$ and such that $\Per_\rho(E)=\Per_\rho(E_1)+\Per_\rho(E_2)$.
\end{defi}
Our goal is to characterize decomposability in terms of the radius of the support of $\rho$. To this aim, we assume in the following that $\rho$ satisfies \ref{itm:h1}-\ref{itm:upper} and $\supp \rho = \overline{B_\epsilon(0)}$ with $\epsilon \in (0,\infty)$. Moreover, we assume that 
\begin{equation}\label{eq:assumptionstrict}
\text{$f_\rho$ in \ref{itm:h1} is decreasing on $(0,\epsilon)$.}
\end{equation}
We first prove the following lemma.
\begin{lemma}\label{le:divergencenonlocalgrad}
    Let $E$ be a finite Caccioppoli $\rho$-perimeter set. Then, $D_\rho \mathds{1}_E$ is a smooth function on $E \setminus \partial E$. Moreover, for all $x \in E \setminus \partial E$ such that $|B_\epsilon(x) \cap E^c| >0$, we have that
    \[
    \div D_\rho \mathds{1}_E (x) < 0.
    \]
\end{lemma}
\begin{proof}
    The smoothness is clear, since for $x \in E \setminus \partial E$, we have that
    \[
    D_\rho \mathds{1}_E(x) = \int_{\R^d}-\mathds{1}_{E^c}(y) \frac{(y-x)\rho(y-x)}{|y-x|^2}\dd{y} = -\int_{E^c} \frac{(y-x)\rho(y-x)}{|y-x|^2}\dd{y},
    \]
    which is smooth by \ref{itm:derivatives} and the fact that $\dist(x,E^c)>0$. Moreover, to compute the divergence, we may interchange it with the integral to obtain
    \[
    \div D_\rho \mathds{1}_E (x) = - \int_{E^c} \div_x \left( \frac{(y-x)\rho(y-x)}{|y-x|^2}\right)\dd{y}=\int_{E^c} \frac{f_\rho'(|y-x|)}{|y-x|^{d-1}}\dd{y}.
    \]
    This last identity follows from a direct computation by using that
    \[
    \frac{(y-x)\rho(y-x)}{|y-x|^2} = f_\rho(|y-x|)\frac{y-x}{|y-x|^d}.
    \]
    Since $f'_\rho$ is negative on $(0,\epsilon)$ by \eqref{eq:assumptionstrict} and $|B_\epsilon(x) \cap E^c| >0$, we conclude that $\div D_\rho \mathds{1}_E (x) < 0.$
\end{proof}
We also need the following lemma for sets with additional regularity, illustrated in Figure \ref{fig:normalsquare}.

\begin{figure}[ht]
    \centering
    \includegraphics[width=0.35\textwidth]{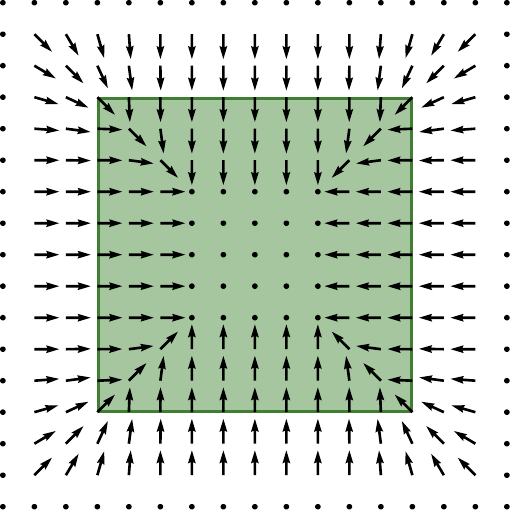}
    \caption{Numerical illustration of the normalized nonlocal gradient $D_\rho \mathds{1}_E / \abs{D_\rho \mathds{1}_E}$ with $E = (-5/4,5/4)^2$ and $\rho(\cdot)=\operatorname{exp}\left(\dfrac{1/10}{|\cdot|^2-(3/4)^2}\right)\1_{B_{3/4}(0)}(\cdot)\dfrac{1}{|\cdot|^{3/2}}.$ The computation was performed at locations $x$ on a grid with spacing $1/4$, with the integral for each of them approximated by a finite sum on a grid of $161 \times 161$ points. The locations where $D_\rho \mathds{1}_E$ vanishes are shown as dots.}\label{fig:normalsquare}
\end{figure}

\begin{lemma}\label{le:unitnormal}
    Let $E$ be a finite perimeter set and $x_0 \in \partial^*E$ be such that the generalized inner unit normal $\nu_E$ restricted to $\partial^*E$ is continuous at $x_0$. Then, for any sequence $(x_n)_n \subset \R^d \setminus \overline{\partial^* E}$ such that $x_n \to x_0$, it holds that
    \[
    \lim_{n \to \infty} \frac{D_\rho \mathds{1}_E(x_n)}{\abs{D_\rho \mathds{1}_E(x_n)}} = \nu_E(x_0).
    \]
\end{lemma}
\begin{proof}
    We can compute with the help of Lemma~\ref{le:fromnonlocaltolocal}, that 
    \begin{align}
    \begin{split}\label{eq:splitboundary}
    D_\rho \mathds{1}_E(x_n) &= (Q_\rho * D\mathds{1}_E)(x_n)=(Q_\rho * \nu_E \Hcal^{d-1}\mres {\partial^* E})(x_n)\\
    &=\int_{\partial^* E} Q_\rho(x_n-y)\nu_E(y)\dd\Hcal^{d-1}(y) \\
    &= \int_{\partial^* E} Q_\rho(x_n-y)\dd\Hcal^{d-1}(y) \,\nu_E(x_0) \\
    &\qquad+ \int_{\partial^* E} Q_\rho(x_n-y)(\nu_E(y)-\nu_E(x_0))\dd\Hcal^{d-1}(y).
    \end{split}
    \end{align}
    Using \ref{itm:lower}, we find that $Q_\rho(z) \geqslant C|z|^{-d+1-\sigma}$ for all $z \in B_\eta(0)$. This yields
    \[
    a_n:=\int_{\partial^* E} Q_\rho(x_n-y)\dd\Hcal^{d-1}(y) \geqslant C \int_{\partial^* E \cap B_{\eta/2}(x_0)} |x_n-y|^{-d+1-\sigma}\dd\Hcal^{d-1}(y),
    \]
    when $x_n \in B_{\eta/2}(x_0)$, which implies $\lim_{n \to \infty}a_n=\infty$ given that $x_n \to x_0 \in \partial^* E$. For the other term, we note that due to the continuity of the unit normal, $\abs{\nu_E(y)-\nu_E(x_0)}\leqslant \eta_k$ for all $y \in \partial^* E \cap B_{1/k}(x_0)$ with $\eta_k \to 0$ as $k \to \infty$. Hence, we can estimate
    \begin{align*}
    b_n&:= \left|\int_{\partial^* E} Q_\rho(x_n-y)(\nu_E(y)-\nu_E(x_0))\dd\Hcal^{d-1}(y)\right| \\
    &\leqslant \eta_k\int_{\partial^* E \cap B_{1/k}(x_0)} Q_\rho(x_n-y)\dd\Hcal^{d-1}(y) \\
    &\qquad + 2\int_{\partial^* E \setminus B_{1/k}(x_0)} Q_\rho(x_n-y)\dd\Hcal^{d-1}(y) \\
    &\leqslant \eta_k a_n + 2\int_{\partial^* E \setminus B_{1/k}(x_0)} Q_\rho(x_n-y)\dd\Hcal^{d-1}(y).
    \end{align*}
    Since the latter integral is bounded uniformly in $n$, we obtain that $\lim_{n \to \infty} b_n/a_n \leqslant \eta_k$ for all $k \in \N$. Therefore, we find that $\lim_{n\to \infty} b_n/a_n = 0$. Using \eqref{eq:splitboundary}, we obtain unit vectors $(v_n)_n \subset \mathbb{S}^{d-1}$ such that
    \[
    D_\rho\mathds{1}_E(x_n) = a_n\nu_E(x_0) + b_n v_n.
    \]
    Finally, we deduce
    \begin{align*}
    \lim_{n \to \infty} \frac{D_\rho \mathds{1}_E(x_n)}{\abs{D_\rho \mathds{1}_E(x_n)}} &=\lim_{n \to \infty} \frac{a_n\nu_E(x_0) + b_n v_n}{\abs{a_n\nu_E(x_0) + b_n v_n}} \\
    &=\lim_{n \to \infty} \frac{\nu_E(x_0) + \frac{b_n}{a_n} v_n}{\abs{\nu_E(x_0) + \frac{b_n}{a_n} v_n}} = \nu_E(x_0).\qedhere
    \end{align*}
\end{proof}
With this we can prove the following indecomposability result, at least when we are working with $C^1$-domains.
\begin{thm}\label{thm:epsdecomp}
    Let $E$ be a finite Caccioppoli $\rho$-perimeter set and $E_1,E_2$ be non-empty bounded $C^{1}$-domains with $E=E_1 \cup E_2$ and $E_1 \cap E_2 = \emptyset$. Then, $\Per_\rho(E)<\Per_\rho(E_1)+\Per_\rho(E_2)$ if and only if $\dist(E_1,E_2)< 2 \epsilon$.
\end{thm}
\begin{proof}
Since $\supp D_\rho \mathds{1}_{E_i} \subset \overline{E_i+B_\epsilon(0)}$ for $i=1,2$, it is clear that $\Per_\rho(E)=\Per_\rho(E_1)+\Per_\rho(E_2)$ if $\dist(E_1,E_2)\geqslant  2 \epsilon$. For the converse, we suppose for the sake of contradiction that $\Per_\rho(E)=\Per_\rho(E_1)+\Per_\rho(E_2)$. We distinguish two cases.

\textit{Case 1: $\dist(E_1,E_2)< \epsilon$.} By Lemma~\ref{le:fromnonlocaltolocal} it holds that $D_\rho \mathds{1}_{E_1}=Q_\rho* D \mathds{1}_{E_1}$ and $D_\rho \mathds{1}_{E_2}=Q_\rho * D \mathds{1}_{E_2}$ are $L^1$-functions so that
\[
\Per_\rho(E)=\int_{\R^d} |D_\rho \mathds{1}_E|\dd{x} = \int_{\R^d} |D_\rho \mathds{1}_{E_1}+D_\rho \mathds{1}_{E_2}|\dd{x}=\int_{\R^d} |D_\rho \mathds{1}_{E_1}|+|D_\rho \mathds{1}_{E_2}|\dd{x}
\]
can only hold if $D_\rho \mathds{1}_{E_1}$ and $D_\rho \mathds{1}_{E_2}$ point in the same direction almost everywhere. It follows from Lemma~\ref{le:unitnormal} that $D_\rho \mathds{1}_{E_i}/|D_\rho \mathds{1}_{E_i}|(x') \to \nu_{E_i}(x)$ as $x' \to x \in \partial E_i$ with $\nu_{E_i}$ the inner unit normal to $\partial E_i$. Hence, $E_1$ and $E_2$ cannot share a boundary point, since then the normals point in different directions. We infer that $\dist(E_1,E_2)>0$, which shows that $D_\rho \mathds{1}_{E_2}$ is smooth on $E_1$. Since we additionally have that $D_\rho \mathds{1}_{E_2}= g\nu_{E_1}$ at $\partial E_1$ for some nonnegative $g$ due to the direction of $D_\rho \mathds{1}_{E_1}$ there, we obtain by the divergence theorem that
\[
\int_{E_1} \div(D_\rho \mathds{1}_{E_2})\dd{x} = \int_{\partial E_1} D_\rho \mathds{1}_{E_2} \cdot (-\nu_{E_1}) \dd{\Hcal^{d-1}} = -\int_{\partial E_1} g \dd{\Hcal^{d-1}} \leqslant 0.
\]
However, the left-hand side of this equation is positive by Lemma~\ref{le:divergencenonlocalgrad} (applied to $\mathds{1}_{E_2^c}$) and the fact that $\dist(E_1,E_2)<\epsilon$. This yields the desired contradiction. \smallskip

\textit{Case 2: $\epsilon \leqslant \dist(E_1,E_2) < 2\epsilon$.} Since $D_\rho \mathds{1}_{E_1}$ and $D_\rho \mathds{1}_{E_2}$ point in the same direction, we have, in particular, that
\[
0 \leqslant \int_{\R^d} D_\rho \mathds{1}_{E_1} \cdot D_\rho \mathds{1}_{E_2}\dd{x} = -\int_{E_2} \div_\rho D_\rho \mathds{1}_{E_1}\dd{x},
\]
with the second identity using integration by parts \eqref{eq:intbyparts}. We note that this is valid since $D_\rho \mathds{1}_{E_1}$ is smooth on $E_2+B_\epsilon(0)$ by Lemma~\ref{le:divergencenonlocalgrad}, which is the set where $D_\rho \mathds{1}_{E_2}$ is nonzero. For every $x \in E_2$, we now have that $\dist(x,E_1)>\epsilon$ so by Lemma~\ref{le:divergencenonlocalgrad}
\[
\div_\rho D_\rho \mathds{1}_{E_1}(x) = (Q_\rho * \div D_\rho \mathds{1}_{E_1})(x) \geqslant  0.
\]
In fact, for all $x \in E_2$ with $\dist(x,E_1) < 2 \epsilon$, the latter quantity is positive, so we deduce
\[
\int_{E_2} \div_\rho D_\rho \mathds{1}_{E_1}\dd{x} > 0,
\]
which yields a contradiction.
\end{proof}
\begin{rem}
    Theorem~\ref{thm:epsdecomp} shows that a set $E$ of finite Caccioppoli $\rho$-perimeter is $\Per_\rho$-decom\-posable into two $C^1$-sets $E_1,E_2$ if and only if
    \[
    \dist(E_1,E_2)=\dist^e(E_1,E_2) \geqslant 2 \epsilon.
    \]
    This is exactly the same characterization obtained for the Gagliardo perimeter in Proposition~\ref{prop:charepsilon}, albeit with $\epsilon$ replaced by $2 \epsilon$. This discrepancy is due to the fact that the nonlocal gradients of $\mathds{1}_{E_1}$ and $\mathds{1}_{E_2}$ are supported on $\epsilon-$neighborhoods of each respective set, meaning that interactions can take place when the distance is less than $2\epsilon$. A remaining open question is whether the $C^1$-smoothness assumption on $E_1$ and $E_2$ can be removed, which would yield a complete characterization of decomposability similar to Proposition~\ref{prop:charepsilon}.
\end{rem}

We finish this section with the decomposability result for the fractional Caccioppoli perimeter $\Per_\alpha$ of Example~\ref{ex:frac}. Since in this case $\epsilon =\infty$, we have that decomposability is never possible. The proof is ommitted since it is completely analogous to the case $\epsilon<\infty$.
\begin{prop}
    Let $E$ be a finite Caccioppoli $\alpha$-perimeter set and $E_1,E_2$ be non-empty bounded $C^{1}$-domains with $E=E_1 \cup E_2$ and $E_1 \cap E_2 = \emptyset$. Then, $\Per_\alpha(E)<\Per_\alpha(E_1)+\Per_\alpha(E_2)$.
\end{prop}

\subsection{Extreme points in \texorpdfstring{$\BV^\rho(\R^d)$}{BVrho} and \texorpdfstring{$\BV^\alpha(\R^d)$}{BValpha}.}\label{sec:distrExt}

In this section we characterize the extreme points of the unit balls of $\BV^\rho(\R^d)$ and $\BV^\alpha(\R^d)$. It turns out that they are not related to the notion of $\Per_\rho$-decomposability, but instead related to the extreme points of $\BV(\R^d)$.

With the tools from Section~\ref{sec:distrBV}, we can immediately provide a characterization of the extreme points in $\BV^\rho(\R^d)$ with respect to the nonlocal total variation seminorm. We assume that $\rho$ has compact support and satisfies \ref{itm:h1}-\ref{itm:upper}. Moreover, we introduce the notation
\[
\Bcal_{\TV_\rho}:=\{u \in \BV^\rho(\R^d)\,|\, \TV_\rho(u) \leqslant1\} \quad \text{and} \quad \Bcal_{\TV}:=\{v \in \BV(\R^d)\,|\, \TV(v)\leqslant1\}.
\]
The extreme points of these balls can be related to each other using the following trivial observation.
\begin{lemma}
    Let $X_1,X_2$ be two vector spaces and $\Phi:X_1 \to X_2$ a linear bijection. Then, for any $\Acal \subset X$ it holds that
    \[
    \Phi(\Ext(\Acal)) = \Ext(\Phi(\Acal)).
    \]
\end{lemma}
Given the fact that $\Pcal_\rho:\BV(\R^d) \to \BV^\rho(\R^d)$ is an isomorphism with $\Pcal_\rho(\Bcal_{\TV})=\Bcal_{\TV_\rho}$, see Lemma~\ref{le:fromnonlocaltolocal} and \ref{le:fromlocaltononlocal}, we obtain the following together with the characterization 
\[
\Ext(\Bcal_{\TV})=\left\{\pm\frac{\mathds{1}_E}{\Per(E)}\,\middle|\, E \in \Scal\right\} \quad \text{with} \ \Scal:=\left\{E \subset \R^d\,\middle|\, E \ \text{simple},  \ |E| \in (0,\infty)\right\},
\]
see~\cite[Proposition~8]{AmbCasMasMor01} and \cite{Fle57}, \cite{fleming60}. Here, simple means that $E$ has finite perimeter and $E$ is indecomposable and saturated. 
\begin{thm}\label{thm:extepsdistributional}
    It holds that
    \[
    \Ext(\Bcal_{\TV_\rho}) = \Pcal_\rho(\Ext(\Bcal_{\TV})),
    \]
    or, more explicitly,
    \begin{align*}
    \Ext(\Bcal_{\TV_\rho})&=\left\{\Pcal_\rho \left(\pm\frac{\mathds{1}_{E}}{\Per(E)}\right) \,\middle|\, E \in \Scal \right\}\\
    &=\left\{u \in \Bcal_{\TV_\rho} \,\middle|\, D_\rho u = \pm\frac{\nu_E}{\Per(E)}\,\Hcal^{d-1}\mres{\partial^*E}, \  E \in \Scal \right\}
    \end{align*}
    with $\partial^*E$ the reduced boundary of $E$, and $\nu_E$ the generalized inner unit normal.
\end{thm}
\begin{rem}\label{rem:extrho}
In the case that the reduced boundary of $E$ is a bounded set, we find by Proposition~\ref{prop:nftoc} that
\[
\Pcal_\rho \left(\pm\frac{\mathds{1}_{E}}{\Per(E)}\right)(x) = \pm\frac{1}{\Per(E)} \int_{\partial^*E} V_\rho(x-y) \cdot \nu_{E}(y)\dd \Hcal^{d-1}(y) \quad \text{for a.e.~$x \in \R^d$,}
\]
which gives a representation of the extreme point in terms of a lower dimensional integral.
\end{rem}

In the fractional case, the operator $(-\Delta)^{\frac{1-\alpha}{2}}:\BV(\R^d) \to \BV^{\alpha}(\R^d)$ is not an isomorphism, cf.~Remark~\ref{rem:fractionalcase}\,(i). However, we can still prove the same result using truncation arguments. 
\begin{thm}\label{thm:extalphadistributional}
    It holds that
    \begin{align*}
    \Ext(\Bcal_{\TV_\alpha})&= \left\{(-\Delta)^{\frac{1-\alpha}{2}} \left(\pm\frac{\mathds{1}_{E}}{\Per(E)}\right) \,\middle|\, E \in \Scal \right\}\\
    &=\left\{u \in \Bcal_{\TV_\alpha} \,\middle|\, D^\alpha u = \pm\frac{\nu_E}{\Per(E)}\,\Hcal^{d-1}\mres{\partial^*E}, \  E \in \Scal \right\}
    \end{align*}
    with $\partial^*E$ the reduced boundary of $E$, and $\nu_E$ the generalized inner unit normal.
\end{thm}
\begin{proof}
\textit{Sufficiency.} Let $E \in \Scal$, $v:=\mathds{1}_{E}/\Per(E) \in \BV(\R^d)$ and 
    \[
    u:=(-\Delta)^{\frac{1-\alpha}{2}}v= (-\Delta)^{\frac{1-\alpha}{2}} \left(\pm\frac{\mathds{1}_{E}}{\Per(E)}\right) \in \BV^{\alpha}(\R^d).
    \]
    Additionally, let $\lambda \in (0,1)$ and $u_1,u_2 \in \Bcal_{\TV_\alpha}$ with $u=\lambda u_1+(1-\lambda)u_2$. If we set $v_1:=I_{1-\alpha}*u_1 \in \BV_{\rm loc}(\R^d)$ and $v_2:=I_{1-\alpha}*u_2 \in \BV_{\rm loc}(\R^d)$, we find
    \[
    v=I_{1-\alpha}*u=\lambda v_1+(1-\lambda)v_2, \quad Dv_1 = D^\alpha u_1 \quad \text{and} \quad Dv_2 = D^\alpha u_2,
    \]
    see~Remark~\ref{rem:fractionalcase}\,(i). Additionally, by the Hardy-Littlewood-Sobolev inequality (cf.~\cite[Theorem~V.1]{Ste70}), we deduce that $v_1,v_2 \in L^{p,\infty}(\R^d)$ with $p=d/(d-1+\alpha)$. Hence, we find by the Sobolev embedding in $\BV(\R^d)$ \cite[Theorem~3.47]{AmbFusPal00}, that
    \[
    v,v_1,v_2 \in \Bcal_{\rm FV}:=\{w \in L^{1^*}(\R^d) \,|\, |Dw|(\R^d)\leq 1\}.
    \]
    with $1^*=d/(d-1)$ for $d>1$ and $1^*=\infty$ for $d=1$. Since $v$ is also an extreme point of $\Bcal_{\rm FV}$ by \cite[Proposition~4.4]{BonGus22}, we infer that $v=v_1=v_2$, and hence,
    \[
    u_1 = (-\Delta)^{\frac{1-\alpha}{2}}v_1 = (-\Delta)^{\frac{1-\alpha}{2}}v_2 = u_2.
    \]
    This proves that $u \in \Ext(\Bcal_{\TV_\alpha})$. \smallskip

    \textit{Necessity.} Let $u \in \Ext(\Bcal_{\TV_\alpha})$ and define $v:=I_{1-\alpha}*u \in \BV_{\rm loc}(\R^d)$, which satisfies $Dv=D^\alpha u$. We suppose without loss of generality that $|\{v>0\}| >0$, since otherwise we can consider $-v$. For $t>0$ small with $|\{v > t\}| >0$, we consider the (non-constant) function $v_t:=\max\{v-t,0\}$. Given the fact that $v \in L^{p,\infty}(\R^d)$ with $p=d/(d-1+\alpha)$, due to the Hardy-Littlewood-Sobolev inequality (cf.~\cite[Theorem~V.1]{Ste70}), we find that
    \begin{align*}
        \|v_t\|_{L^1(\R^d)} &\leqslant\|\mathds{1}_{\{v > t\}}v\|_{L^1(\R^d)} = \|\mathds{1}_{\{v > t\}}v\|_{L^{1,1}(\R^d)}\\
        & \leqslant C \|\mathds{1}_{\{v > t\}}\|_{L^{p',1}(\R^d)}\|v\|_{L^{p,\infty}(\R^d)}=Cp'|\{v>t\}|^{1/p'}\|v\|_{L^{p,\infty}(\R^d)} < \infty,
    \end{align*}
    where we have used H\"{o}lder's inequality on the scale of Lorentz spaces. We conclude that $v_t \in L^1(\R^d)$. Moreover, by \cite[Lemma~3.2]{BonGus22}, we find that $v_t \in \BV(\R^d)$ with
    \begin{equation}\label{eq:splitting}
    1=|D^\alpha u|(\R^d) = |Dv|(\R^d) = |Dv_t|(\R^d) + |D(v-v_t)|(\R^d).
    \end{equation}
    We now define with $\lambda :=|Dv_t|(\R^d) \in (0,1)$, the functions
    \[
    u_1 = \frac{1}{\lambda}(-\Delta)^{\frac{1-\alpha}{2}}v_t \quad \text{and} \quad u_2:=\frac{1}{1-\lambda}(u-\lambda u_1),
    \]
    which both lie in $\BV^\alpha(\R^d)$. It holds that $u=\lambda u_1+(1-\lambda)u_2$ and $u_1,u_2 \in \Bcal_{\TV_\alpha}$ by \eqref{eq:splitting}. From $u \in \Ext(\Bcal_{\TV_\alpha})$, we conclude that $u=u_1$, or equivalently, $v=v_t/\lambda$. This is only possible if $v$ is constant on $\{v >t\}$ and zero outside of this set, that is,
    \[
    v = \sigma \mathds{1}_{\{v>t\}} \ \text{for some $\sigma >0$.}
    \]
    This shows that $v \in \BV(\R^d)$, since the set $\{v >t\}$ has finite measure from the observation that $v \in L^{p,\infty}(\R^d)$. It now follows that $v\in \Ext(\Bcal_{\TV})$. Indeed, if $v=\lambda v_1 + (1-\lambda)v_2$ with $v_1,v_2 \in \Bcal_{\TV}$ satisfying $v_1\not =v_2$, then we find that
    \[
    u = \lambda (-\Delta)^{\frac{1-\alpha}{2}}v_1 + (1-\lambda)(-\Delta)^{\frac{1-\alpha}{2}}v_2,
    \]
    which contradicts $u \in \Ext(\Bcal_{\TV_\alpha})$. Therefore, it holds that $v = \pm \mathds{1}_E/\Per(E)$ for some $E \in \Scal$, from which we conclude that 
    \[
    u = (-\Delta)^{\frac{1-\alpha}{2}}v=(-\Delta)^{\frac{1-\alpha}{2}}\left(\pm\frac{\mathds{1}_{E}}{\Per(E)}\right)
    \]
    as desired.\qedhere    
\end{proof}
\begin{rem}
\leavevmode
\begin{itemize}\label{rem:ext1d}
    \item[(i)] If the reduced boundary of $E$ is bounded, we find by Remark~\ref{rem:fractionalcase}\,(ii) that
    \[
    (-\Delta)^{\frac{1-\alpha}{2}} \left(\pm\frac{\mathds{1}_{E}}{\Per(E)}\right)(x)=\pm\frac{c_{d,-\alpha}}{\Per(E)} \int_{\partial^*E} \frac{x -y}{|x -y|^{d+1-\alpha}} \cdot \nu_{E}(y)\dd \Hcal^{d-1}(y) \quad \text{for a.e.~$x \in \R^d$.}
    \]

\item[(ii)] In the case where $d=1$, the collection of simple sets is given by intervals $(a,b)$ with $a<b$. Hence, Theorem~\ref{thm:extalphadistributional} and part (i) of this remark yields that
\[
\Ext(\Bcal_{\TV_\alpha})=\left\{ x \mapsto \pm \frac{c_{1,-\alpha}}{2} \left(\frac{x-a}{|x - a|^{2-\alpha}}\ - \frac{x-b}{|x-b|^{2-\alpha}}\right) \,\middle|\, a<b\right\}.
\]
This shows that the extreme points in $\BV^\alpha$ are not indicator functions, nor do they need to have compact support, cf.~Figure~\ref{fig:uab} for an illustration. \qedhere
\end{itemize}
\end{rem}
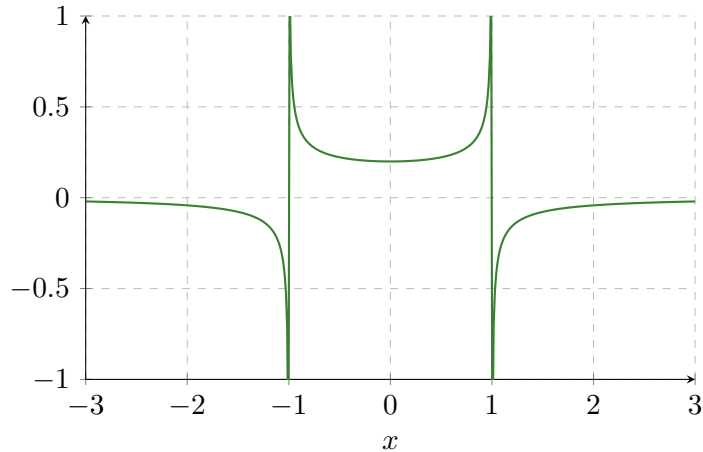
\begin{figure}[ht]
    \centering
\begin{tikzpicture}
\begin{axis}[
    height=0.4\textwidth,
    width=0.6\textwidth,
    axis lines = left,    
    xlabel = \(x\),
    grid = major,
    ymin=-1,
    ymax=1,
]
\addplot [
    domain=-3:3, 
    samples=500, 
    color=OliveGreen,
    thick,
    point meta={0.0997*abs((x+1)/abs(x+1)^(3/2) - (x-1)/abs(x-1)^(3/2))< 3 ? nan : y},
    ] 
{0.0997*((x+1)/abs(x+1)^(3/2) - (x-1)/abs(x-1)^(3/2))};
\end{axis}
\end{tikzpicture}
\caption{Plot of the one-dimensional extreme point of Remark~\ref{rem:ext1d}\,(ii) with $\alpha=1/2$, $a=-1$ and $b=1$.}\label{fig:uab}
\end{figure}

Using the operators from Remark~\ref{rem:fractionalcase}\,(i), we can also deduce a type of coarea formula, albeit by using a special decomposition. Compare also with the fractional coarea inequality in \cite[Theorem~3.11 and Corollary~5.6]{ComSte19}.
\begin{prop}[Non-standard coarea formula]
    Let $u \in \BV^\alpha(\R^d)$, then it holds that
    \[
    D^\alpha u = \int_{\R} D^\alpha u_t \dd t \quad \text{and} \quad |D^\alpha u| = \int_{\R} |D^\alpha u_t|\dd t=\int_{\R}\Per(E_t)\dd{t},
    \]
    where $u_t:=(-\Delta)^{\frac{1-\alpha}{2}} (\mathds{1}_{E_t}) \in \BV^\alpha(\R^d)$ for a.e.~$t \in \R$ with $E_t:=\{I_{1-\alpha}*u>t\}$. In particular, $D^\alpha u_t = \nu_{E_t}\Hcal^{d-1}\mres{\partial^* E_t}$ for a.e.~$t \in \R$.
\end{prop}
\begin{proof}
    Define the function $v:=I_{1-\alpha}*u \in \BV_{\rm loc}(\R^d)$, which satisfies $Dv=D^\alpha u$. By the classical coarea formula \cite[Theorem~3.40]{AmbFusPal00}, we find that
    \begin{equation}\label{eq:clascoarea}
    D^\alpha u = Dv= \int_{\R} D \mathds{1}_{\{I_{1-\alpha}*u>t\}} \dd t \quad \text{and} \quad |D^\alpha u| = |Dv|= \int_{\R} |D \mathds{1}_{\{I_{1-\alpha}*u>t\}}| \dd t.
    \end{equation}
    Note also that $\mathds{1}_{\{I_{1-\alpha}*u>t\}} \in \BV(\R^d)$ for a.e.~$t >0$ since $I_{1-\alpha}*u \in L^{p,\infty}(\R^d)$ with $p=d/(d-1+\alpha)$. Hence, we can define $u_t:=(-\Delta)^{\frac{1-\alpha}{2}} (\mathds{1}_{\{I_{1-\alpha}*u>t\}}) \in \BV^\alpha(\R^d)$, which satisfies $D^\alpha u_t =D\mathds{1}_{\{I_{1-\alpha}*u>t\}}$. On the other hand, for a.e.~$t<0$, we note that $\mathds{1}_{\{I_{1-\alpha}*u\leqslant t\}} \in \BV(\R^d)$ and hence, also 
    \[
    u_t:=(-\Delta)^{\frac{1-\alpha}{2}} (\mathds{1}_{\{I_{1-\alpha}*u>t\}})=(-\Delta)^{\frac{1-\alpha}{2}} (\mathds{1}_{\R^d}-\mathds{1}_{\{I_{1-\alpha}*u\leqslant t\}})=-(-\Delta)^{\frac{1-\alpha}{2}}(\mathds{1}_{\{I_{1-\alpha}*u\leqslant t\}})
    \]
    is well-defined with $D^\alpha u_t =D\mathds{1}_{\{I_{1-\alpha}*u>t\}}$. The statement now follows from \eqref{eq:clascoarea}.
\end{proof}

\section{Localization of nonlocal Caccioppoli perimeters}\label{sec:gammadistributional}

Here, we prove that the nonlocal Caccioppoli perimeters $\Gamma$-converge to the classical perimeter as the interaction range vanishes. In fact, we will prove this more generally for the nonlocal $\TV$-functional, and subsequently restrict to indicator functions to obtain the analogous statement for the associated perimeters. The results in this section extend the localization from \cite{MenSpe15, cueto2024gamma} for nonlocal Sobolev spaces to the nonlocal bounded variation spaces.

We assume that $\rho$ satisfies \ref{itm:h1}-\ref{itm:upper} and is normalized to $\int_{\R^d} \rho \dd z = d$ and $\supp \rho = \overline{B_1(0)}$. We introduce the rescaled kernels
\[
\rho_\epsilon(z):=\frac{1}{\epsilon^d}\rho(z/\epsilon)
\]
and naturally consider the functionals $\TV_{\epsilon}:L^1_{\rm loc}(\R^d) \to [0,\infty]$ given by
\[
\TV_{\epsilon} (u):=\sup\left\{ \int_{\R^d} u \div_{\rho_\epsilon} p \dd x \,\middle\vert\, p \in C_c^{\infty}(\R^d;\R^d),\, \|p\|_{L^{\infty}(\R^d;\R^d)}\leq1\right\},
\]
which can be equivalently characterized as
\[
\TV_{\epsilon} (u)=\begin{cases}
    |D_{\rho_\epsilon}u|(\R^d) &\text{if $D_{\rho_\epsilon} u \in \Mcal(\R^d;\R^d)$,}\\
    \infty & \text{else}.
    \end{cases}
\]
We have the following localization result in terms of $\Gamma$-convergence, cf.~\cite{braides2002gamma, DM92} for a general introduction to $\Gamma$-convergence.
\begin{thm}\label{thm:localizationgradient}
    It holds that $\TV_{\epsilon}$ $\Gamma$-converges with respect to the $L^1_{\rm loc}(\R^d)$-topology to the functional $\TV$ as $\epsilon \to 0$.
\end{thm}
\begin{proof}
Let $(\epsilon_n)_n \subset (0,\infty)$ be a sequence converging to 0. We split the proof up into two steps.\smallskip

\textit{Liminf-inequality:} Let $u_n \to u$ in $L^1_{\rm loc}(\R^d)$ and suppose without loss of generality that
\[
\sup_n |D_{\rho_{\epsilon_n}}u_n|(\R^d) < \infty.
\]
Then, up to a non-relabeled subsequence, we find that $D_{\rho_{\epsilon_n}}u$ converges weak* to a measure $\mu \in \Mcal(\R^d;\R^d)$. This implies for any $p \in C_c^{\infty}(\R^d;\R^d)$ that
\begin{align*}
    \int_{\R^d} p \cdot \dd{\mu} &= \lim_{n \to \infty} \int_{\R^d} p \cdot \dd{D_{\rho_{\epsilon_n}}u_n} \\
    &= - \lim_{n \to \infty} \int_{\R^d} u_n \div_{\rho_{\epsilon_n}} p \dd{x} \\
    &= -\int_{\R^d} u \div p \dd{x},
\end{align*}
where the last line uses that $\div_{\rho_{\epsilon_n}} p \to \div p$ uniformly, see~\cite[Lemma~3.1\,$(i)$]{cueto2024gamma}, together with the fact that their supports are contained in a fixed compact set. We conclude that $u \in \BV_{\rm loc}(\R^d)$ and $Du=\mu$. The weak* lower semicontinuity of the norm on $\Mcal(\R^d;\R^d)$ now yields
\[
\TV(u) = |Du|(\R^d) \leqslant \liminf_{n \to \infty} |D_{\rho_{\epsilon_n}}u_n|(\R^d)=\liminf_{n \to \infty}\TV_{{\epsilon_n}}(u_n).
\]

\textit{Recovery-sequence:} Let $u \in \BV_{\rm loc}(\R^d)$ with $\TV(u)<\infty$. Then, a simple argument using Fubini's theorem and integration by parts shows that $D_{\rho_{\epsilon_n}}u \in \Mcal(\R^d;\R^d)$ with 
\[
D_{\rho_{\epsilon_n}} u = Q_{\rho_{\epsilon_n}}*Du.
\]
From \cite[Eq.~(2.16)]{cueto2024gamma}, we find that $Q_{\rho_{\epsilon_n}}(z)=\epsilon_n^{-d}Q_\rho(z/\epsilon_n)$ and $\|Q_{\rho_{\epsilon_n}}\|_{L^1(\R^d)}=1$ for all $n \in \N$. Hence, we find that
\[
\limsup_{n \to \infty} |D_{\rho_{\epsilon_n}} u|(\R^d) \leqslant \limsup_{n \to \infty} \|Q_{\rho_{\epsilon_n}}\|_{L^1(\R^d)}|D u|(\R^d)=|D u|(\R^d).
\]
On the other hand, arguing as for the liminf-inequality, we have that $D_{\rho_{\epsilon_n}} u$ converges weak* to $Du$, so that
\[
|Du|(\R^d) \leqslant \liminf_{n \to \infty} |D_{\rho_{\epsilon_n}}u|(\R^d).
\]
Combining the previous two estimates shows that
\[
\lim_{n \to \infty} \TV_{{\epsilon_n}}(u) = \TV(u),
\]
which implies that the constant sequence constitutes a recovery sequence.
\end{proof}
Along with this $\Gamma$-convergence result, we have the following compactness statement.
\begin{thm}\label{thm:compactness}
    Let $\epsilon_n \to 0$ and $(u_n)_n$ be a sequence in $\BV^{\rho_{\epsilon_n}}(\R^d)$ with $\supp u_n \subset B_R(0)$ for some $R>0$, and
    \begin{equation}\label{eq:tvepbounds}
    \sup_{n}|D_{\rho_{\epsilon_n}}u_n|(\R^d) < \infty.
    \end{equation}
    Then, up to a non-relabeled subsequence, $u_n \to u$ in $L^1(\R^d)$.
\end{thm}
\begin{proof}
Since we can find via mollification a sequence $(v_n)_n \subset C_c^{\infty}(\R^d)$ with $\|u_n-v_n\|_{L^1(\R^d)} \leqslant 1/n$, and still satisfying \eqref{eq:tvepbounds}, we may assume without loss of generality that $(u_n)_n \subset C_c^{\infty}(\R^d)$. Now, by rescaling the fundamental theorem of calculus from Theorem~\ref{thm:nftocsmooth}, we find that
\begin{equation}\label{eq:ftocepsilon}
u_n(x) = \int_{\R^d} D_{\rho_{\epsilon_n}}u_n(y)\cdot V_{\rho_{\epsilon_n}}(x-y)\dd{y} \quad \text{with $V_{\rho_{\epsilon_n}}(z):=\frac{1}{\epsilon^{d-1}}V_{\rho}(z/\epsilon)$.}
\end{equation}
By Lemma~\ref{le:Vrhoestimates}, we find that
\[
|V_{\rho_{\epsilon_n}}(z)|+|z||\nabla V_{\rho_{\epsilon_n}}(z)| \leqslant C\max\left\{\epsilon_n^{1-\sigma}\frac{1}{|z|^{d-\sigma}},\frac{1}{|z|^{d-1}}\right\},
\]
with $C$ independent of $n$. In particular, there is a constant $C_R>0$ independent of $n$ such that
\[
|V_{\rho_{\epsilon_n}}(z)|+|z||\nabla V_{\rho_{\epsilon_n}}(z)| \leqslant C_R \frac{1}{|z|^{d-\sigma}} \quad\text{for all $z \in B_{2R+3}(0)\setminus\{0\}$}.
\]
Arguing as in \cite[Lemma~6.4]{BelMorSch24}, we obtain for $|\zeta|\leqslant 1$ that
\[
\|V_{\rho_{\epsilon_n}}(\cdot)-V_{\rho_{\epsilon_n}}(\cdot+\zeta)\|_{L^1(B_{2R+2}(0))} \leqslant C_R|\zeta|^{\sigma}.
\]
Hence, if we assume without loss of generality that $\epsilon_n \leqslant 1$ for all $n \in \N$ so that $\supp D_{\rho_{\epsilon_n}}u_n \subset B_{R+1}(0)$, then we get from \eqref{eq:ftocepsilon} that
\[
|u_n(x)-u_n(x+\zeta)| \leqslant \int_{B_{2R+2}(0)}|V_{\rho_{\epsilon_n}}(y)-V_{\rho_{\epsilon_n}}(y+\zeta)||D_{\rho_{\epsilon_n}}u_n(x-y)|\dd{y}
\]
for all $x \in B_{R+1}(0)$. Using Young's convolution inequality yields
\begin{align*}
    \|u_n(\cdot)-u_n(\cdot+\zeta)\|_{L^1(\R^d)}&=\|u_n(\cdot)-u_n(\cdot+\zeta)\|_{L^1(B_{R+1}(0))}\\
    &\leqslant \|V_{\rho_{\epsilon_n}}(\cdot)-V_{\rho_{\epsilon_n}}(\cdot+\zeta)\|_{L^1(B_{2R+2}(0))}\|D_{\rho_{\epsilon_n}}u_n\|_{L^1(\R^d)} \\
    &\leqslant C_R|\zeta|^{\sigma}\|D_{\rho_{\epsilon_n}}u_n\|_{L^1(\R^d)}.
\end{align*}
Consequently, we obtain that
\[
\lim_{\zeta \to 0} \sup_{n}\|u_n(\cdot)-u_n(\cdot+\zeta)\|_{L^1(\R^d)} =0,
\]
so that we can conclude the result by the Fr\'echet-Kolmogorov criterion.
\end{proof}

Because the recovery sequence in the proof of Theorem~\ref{thm:localizationgradient} is constant, we immediately find that the $\Gamma$-convergence remains valid if we restrict $\TV_{{\epsilon}}$ and $\TV$ to a closed subset of $L^1_{\rm loc}(\R^d)$. In particular, we can restrict to indicator functions to get convergence of the nonlocal Caccioppoli perimeters to the classical perimeter, and also consider boundary conditions and a mass constraint to exploit the compactness in Theorem~\ref{thm:compactness}. Precisely, we consider the functionals
\[
\Per_{\epsilon}:\Mcal_d \to [0,\infty], \qquad \Per_{\epsilon}(E) = \TV_{\epsilon}(\mathds{1}_E)
\]
and for $\Omega \subset \R^d$ an open and bounded set and $m \in (0,|\Omega|)$, the functionals
\[
\Fcal_\epsilon:\Mcal_d \to [0,\infty], \qquad \Fcal_\epsilon(E)=\begin{cases}
    \Per_{\epsilon}(E) & \text{if $E \subset \Omega$ and $|E|=m$,}\\
    \infty & \text{else,}
\end{cases}
\]
and
\[
\Fcal:\Mcal_d \to [0,\infty], \qquad \Fcal(E)=\begin{cases}
    \Per(E) & \text{if $E \subset \Omega$ and $|E|=m$,}\\
    \infty & \text{else.}
\end{cases}
\]
\begin{cor}\label{cor:perimeterlocalization}
    The following two statements hold:
    \begin{itemize}
        \item[(i)] The sequence $(\Per_{\epsilon})_\epsilon$ $\Gamma$-converges as $\epsilon \to 0$ to $\Per$ with respect to the local convergence of sets.
        \item[(ii)] The sequence $(\Fcal_\epsilon)_\epsilon$ $\Gamma$-converges as $\epsilon \to 0$ to $\Fcal$ with respect to the convergence of sets. Moreover, the sequence is also equi-coercive in this topology.
    \end{itemize}
\end{cor}
\begin{proof}
    The $\Gamma$-convergence in (i) and (ii) is a direct consequence of Theorem~\ref{thm:localizationgradient}, by using that the recovery sequence can still be chosen to be constant. For the equi-coercivity in (ii), we note that if $(E_\epsilon)_\epsilon \subset \Mcal_d$ is a sequence of measurable sets with $E_\epsilon \subset \Omega$ for all $\epsilon$ and
    \[
    \sup_\epsilon \Per_{\epsilon}(E_\epsilon)<\infty,
    \]
    then Theorem~\ref{thm:compactness} shows that $(\mathds{1}_{E_{\epsilon}})_\epsilon$ converges up to subsequence in $L^1(\R^d)$. The limit will again be an indicator function $\mathds{1}_E$ for some $E \subset \Omega$, which shows that $E_\epsilon \to E$. This proves the equi-coercivity.
\end{proof}
As an application of this result, we can study the asymptotics of the isoperimetric problem related to the nonlocal Caccioppoli perimeters. Indeed, up to our knowledge, it is currently an open problem whether minimizers of $\Per_{\epsilon}$ under a mass constraint, or the fractional Caccioppoli perimeter, are actually balls. However, with this $\Gamma$-convergence result we can prove that they must converge to balls as $\epsilon \to 0$.
\begin{cor}
    Let $R>0$, $m \in (0,|B_R(0)|)$ and let $(E_\epsilon)_\epsilon$ be a sequence of sets with $E \subset B_R(0)$ and $|E_\epsilon| =m$ for all $\epsilon >0$ and such that
    \[
    \Per_{\epsilon}(E_\epsilon)=c_{\epsilon}^{m,R}:=\inf\left\{\Per_{\epsilon}(E)\,\middle|\,E \subset B_R(0),\  |E|=m\right\}.
    \]
    Then, up to subsequence, $E_\epsilon$ converges to a ball of measure $m$ in $B_R(0)$ as $\epsilon \to 0$ and it holds that
    \[
    \lim_{\epsilon \to 0} c_{\epsilon}^{m,R} =d|B_1(0)|m^{\frac{d-1}{d}}.
    \]
\end{cor}
\begin{proof}
    We note that $E_\epsilon$ is exactly a minimizer of $\Fcal_\epsilon$ with $\Omega=B_R(0)$, so that Corollary~\ref{cor:perimeterlocalization}\,(ii) and the properties of $\Gamma$-convergence imply that $(E_\epsilon)_\epsilon$ converges up to subsequence to a minimizer $E$ of $\Fcal$ and
    \[
    \lim_{\epsilon \to 0}\Fcal_\epsilon(E_\epsilon) = \Fcal(E).
    \]
    Using that the minimizers of the classical isoperimetric problem are exactly balls, we find that $E$ must be a ball of measure $m$ in $B_R(0)$ after which the result follows.
\end{proof}

\section*{Acknowledgments}
This work is the result of a collaboration which started when all authors were present in the Workshop Calculus of Variations NL 2024 in Schiermonnikoog, for which the authors are grateful. H.S. was funded by the Austrian Science Fund (FWF) projects \href{https://doi.org/10.55776/F65}{10.55776/F65} and \href{https://doi.org/10.55776/Y1292}{10.55776/Y1292}.

\subsection*{Data availability}
No datasets were generated or analyzed during the current study.

\subsection*{Conflict of interest}
All authors declare that they have no conflict of interests related to this publication.

\bibliographystyle{abbrv}
\bibliography{CarGraIglSch25-rev}

\end{document}